\documentclass[a4paper]{amsart}

\usepackage[utf8]{inputenc}
\DeclareUnicodeCharacter{00A0}{ }
\usepackage{amssymb}
\usepackage{MnSymbol}
\usepackage{enumitem}
\usepackage[colorlinks,pdfpagelabels,pdfstartview = FitH,bookmarksopen = true,bookmarksnumbered = true,linkcolor = blue,plainpages = false,hypertexnames = false,citecolor = red] {hyperref}

\usepackage{cleveref}
\crefname{enumi}{}{}
\crefname{equation}{}{}

\allowdisplaybreaks

\title[A Carleson type inequality]{A Carleson type inequality for fully nonlinear elliptic equations with non-Lipschitz drift term}

\author[Avelin]{Benny Avelin}
\address{Benny Avelin,
Aalto University, 
Institute of Mathematics,
P.O. Box 11100, 
FI-00076 Aalto, 
Finland}
\address{Benny Avelin,
Department of Mathematics, 
Uppsala University,
S-751 06 Uppsala, 
Sweden} 
\email{\color{blue} benny.avelin@math.uu.se}
\author[Julin]{Vesa Julin}
\address{Vesa Julin,
Department of Mathematics and Statistics, 
University of Jyv\"askyl\"a,
P.O. Box 35, 
40014 Jyv\"askyl\"a, 
Finland} 
\email{\color{blue} vesa.julin@jyu.fi}

\newtheorem{theorem}{Theorem}[section]

\newtheorem{lemma}{Lemma}[section]

\newtheorem{corollary}{Corollary}[section]
\theoremstyle{definition}
\newtheorem{definition}{Definition}
\newtheorem{remark}{Remark}[section]

\numberwithin{equation}{section}

\def\ep{\varepsilon}
\renewcommand{\varrho}{\rho}

\newcommand{\eps}{\varepsilon}

\def\dist{\operatorname{dist}}

\DeclareMathOperator*{\osc}{osc}

\def\en{\mathbb N}
\def\er{\mathbb R}

\def\rn{\mathbb R^N}

\newtoks\by
\newtoks\paper
\newtoks\book
\newtoks\jour
\newtoks\yr
\newtoks\pages
\newtoks\vol
\newtoks\publ

\def\name[#1, #2]{#1 #2}
\def\ota{{\hbox{\bf ???}}}
\def\cLear{\by=\ota\paper=\ota\book=\ota\jour=\ota\yr=\ota
\pages=\ota\vol=\ota\publ=\ota}
\def\endpaper{\the\by, \textit{\the\paper},
{\the\jour} \textbf{\the\vol} (\the\yr), \the\pages.\cLear}
\def\endbook{\the\by, \textit{\the\book},
\the\publ, \the\yr.\cLear}
\def\endpap{\the\by, \textit{\the\paper}, \the\jour.\cLear}
\def\endproc{\the\by, \textit{\the\paper}, \the\book, \the\publ,
\the\yr, \the\pages.\cLear}

\begin{document} 
\begin{abstract}
	This paper concerns the boundary behavior of solutions of certain fully nonlinear equations with a general drift term. We elaborate on the non-homogeneous generalized Harnack inequality proved by the second author in \cite{Vesku}, to prove a generalized Carleson estimate. We also prove boundary H\"older continuity and a boundary Harnack type inequality. 
\end{abstract}
\maketitle

\section{Introduction} In this paper we study the boundary behavior of solutions of the following non-homogeneous, fully nonlinear equation 
\begin{equation}
	\label{thePDE} F(D^2u, Du, x) = 0\,. 
\end{equation}
The operator $F$ is assumed to be elliptic in the sense that there are $0< \lambda \leq \Lambda$ such that 
\begin{equation}
	\label{ellipticity} \lambda \text{Tr}(Y) \leq F(X, p, x) - F(X +Y, p, x)\leq \Lambda \text{Tr}(Y)\,, \quad \forall (x,p) \in \rn \times \rn 
\end{equation}
for every pair of symmetric matrices $X, Y$ where $Y$ is positive semidefinite. We assume that $F$ has a drift term which satisfies the following growth condition 
\begin{equation}
	\label{non-homogeneity} |F(0, p, x) \big| \leq \phi(|p|)\,, \quad \forall (x,p) \in \rn \times \rn 
\end{equation}
where $\phi: [0,\infty) \to [0,\infty)$ is continuous, increasing, and satisfies the structural conditions from \cite{Vesku} (see \Cref{secprel}). Note that the function $F(\cdot, p ,x)$ is $1$-homogeneous while $F(0,\cdot,x)$ in general is not. In the case there is no drift term, i.e., $\phi=0$, we say that the equation \cref{thePDE} is homogeneous. 

The problem we are interested in is the so-called Carleson estimate \cite{Carl}. The Carleson estimate can be stated for the Laplace equation in modern notation as follows. Let $\Omega \subset \rn$ be a sufficiently regular bounded domain and $x_0 \in\partial \Omega$ . If $u$ is a non-negative harmonic function in $ B(x_0,4R) \cap \Omega$ which vanishes continuously on $\partial \Omega \cap B(x_0,4R)$, then 
\begin{equation}
	\label{lapcarl} \sup_{B(x_0,R/C) \cap \Omega} u \leq C u(A_R)\,, 
\end{equation}
where the constant $C$ depends only on $\partial \Omega$ and $N$, and where $A_R \in B(x_0,R/C) \cap \Omega $ such that $d(A_R,\partial \Omega) > R/C^2$ ($A_R$ is usually called a corkscrew point). For $\Omega$ to be sufficiently regular it is enough to assume that $\Omega$ is e.g., an NTA-domain, see \cite{JK}. The Carleson estimate is very important and useful when studying the boundary behavior and free boundary problems for linear elliptic equations \cite{CFMS,CFS,JK,K}, for $p$-Laplace type elliptic equations \cite{ALuN,ALuN1,AN1,LLuN,LN1,LN5,LN6}, for parabolic $p$-Laplace type equations \cite{A1,AGS}, and for homogeneous fully nonlinear equations \cite{Fe,FS,FS1}. 

In this paper we deal with either Lipschitz or $C^{1,1}$ domains and assume that they are locally given by graphs in balls centered at the boundary with radius up to $R_0 > 0$ which unless otherwise stated satisfies $R_0 \leq 16$. For a given Lipschitz domain with Lipschitz constant $l$ we denote $L = \max\{l,2 \}$. The main result of this paper is the sharp Carleson type estimate for non-negative solutions of \cref{thePDE}. Due to the non-homogeneity of the equation it is easy to see that \cref{lapcarl} cannot hold. Instead the Carleson estimate takes a similar form as the generalized interior Harnack inequality proved in \cite{Vesku} (see \Cref{weakHarnack}). Our main result reads as follows. 
\begin{theorem}
	\label{mainthm} Assume that $\Omega$ is a Lipschitz domain such that $0 \in\partial \Omega$ and assume $u \in C(B_{4R}\cap \overline{\Omega})$, with $R \in (0,R_0/4]$, is a non-negative solution of \cref{thePDE}. Let $A_R \in B_{R/2L} \cap \Omega $ be a point such that $d(A_R,	\partial \Omega) > R/(4L^2)$, and assume that $u = 0$ on $\partial \Omega \cap B_{4 R}$. There exists a constant $C > 1$ which is independent of $u$ and of the radius $R$ such that 
	\begin{equation*}
		\int_{u(A_R)}^{M} \frac{dt}{R^2 \phi(t/R)+t} \leq C\,, 
	\end{equation*}
	where $ M = \sup_{B_{R/C} \cap \Omega} u$. 
\end{theorem}
This result is sharp since already the interior Harnack estimate is sharp. The novelty of \Cref{mainthm} is that the constant does not depend on the solution itself. 

Let us point out a few consequences of \Cref{mainthm}. Let us assume that $u$ is as in the theorem. First, if $\phi$ satisfies 
\begin{equation}
	\label{osgood2} \int_1^{\infty} \frac{dt}{\phi(t)} = \infty 
\end{equation}
then there is an increasing function $f_R$ such that the maximum $M$ is bounded by the value $f_R(u(A_R))$. The function $f_R$ is defined implicitly by the estimate in the theorem. If $\phi$ does not satisfy \cref{osgood2} then the maximum $M$ may take arbitrary large values (see \cite{Vesku}). However, even if $\phi$ does not satisfy \cref{osgood2} we may still deduce that if $u(A_R) \leq 1$ then the maximum $M$ is uniformly bounded assuming that the radius $R$ is small enough. This follows from the fact that $R^2 \phi(t/R) \to 0$ locally uniformly under our growth assumption on $\phi$. Second if $\phi$ satisfies the Osgood condition 
\begin{equation}
	\label{osgood1} \int_0^1 \frac{dt}{\phi(t)} = \infty 
\end{equation}
then $u(A_R) = 0$ implies that $u$ is zero everywhere. In other words \cref{osgood1} implies the strong minimum principle. If $\phi$ does not satisfy \cref{osgood1} then the strong minimum principle does not hold. Finally, in the homogeneous case $\phi = 0$ \Cref{mainthm} reduces to the classical Carleson estimate. 

In the homogeneous case, perhaps the most flexible proof of the Carleson estimate is due to \cite{CFMS} and has been adapted to many situations, see e.g., \cite{AdLu,A1,AGS,CNP,FaSa1,FaSa2}. This proof relies on two basic estimates: 
\begin{enumerate}
	[label=(\arabic*)] 
	\item \label{hold} A decay estimate up to the boundary (H\"older continuity), sometimes denoted by the oscillation lemma. 
	\item \label{blow} An upper estimate of the blow-up rate for singular solutions. 
\end{enumerate}
The point is that the rate of blow-up dictated by \cref{blow} does not need to be sharp, this is because it only needs to match the geometric decay dictated by \cref{hold}.

Let us make some notes regarding the proof of \cref{hold} and \cref{blow} in the homogeneous case. In the context of divergence form equations, the proof of \cref{hold} is standard and follows e.g., from the flexible methods developed by De Giorgi \cite{DeG}, and is thus valid for very general domains (outer density condition). However, in the context of non-divergence form equations, this is far from trivial if the domain is irregular. In fact, in Lipschitz domains it is basically only known for linear equations, and the proof relies on the classical result by Krylov and Safonov \cite{KS1, KS2}. For more regular domains the approach is usually via flattening, symmetry and iterating the Harnack inequality. If the Harnack inequality for a non-negative solution in $B_{2R}$ holds, i.e., 
\begin{equation}
	\nonumber \sup_{B_R} u \leq C \inf_{B_R} u\,, 
\end{equation}
for a constant $C$ independent of $u$ and $R$, then a well known proof of \cref{blow} consists of iterating the Harnack inequality in a dyadic fashion up to the point of singularity. 

Due to the non-homogeneity of our equation the classical Harnack inequality no longer holds, and we will instead use the generalized Harnack inequality, which states that a non-negative solution $u \in C(B_{2R})$ of \cref{thePDE} with $R \leq 1$ satisfies 
\begin{equation}
	\label{px new harnack} \int_{m}^{M} \frac{dt}{R^2\phi(t/R)+ t} \leq C, 
\end{equation}
where $m = \inf_{B_R} u$, $M = \sup_{B_R} u$ and $C$ is a constant which is independent of $u$ and $R$. To continue our discussion it is important to note that the standard Harnack inequality for harmonic functions can be written as 
\begin{equation}
	\nonumber \int_m^M \frac{dt}{t} \leq C. 
\end{equation}
As such, the term $R^2 \phi(t/R)$ in \cref{px new harnack} is the non-homogeneous correction term which compensates the effect of \cref{non-homogeneity}. When using \cref{px new harnack} the ''contest'' between the correction term and the base term $t$ becomes evident. When we study the blow-up rate \cref{blow} for solutions of \cref{thePDE} (\Cref{blow up 2}) our goal is to show that for every solution there exists a critical threshold level where the correction term becomes small and stays small, all the way up to the singularity. This means that the asymptotic behavior after the critical level is the same as in the homogeneous case. This argument strongly relies on the structural assumptions on $\phi$ which imply that $\phi(t)= \eta(t)t$ for a slowly increasing function $\eta$. Similarly when we prove the H\"older continuity estimate (\Cref{holder cont,bdry holder cont}) we show that there is a critical radius such that below it the oscillation of the solution reduces in a geometric fashion. Again the point is to quantify the critical radius. 

\subsection{First application: Boundary Harnack inequality}

In the last section of this paper, we consider the boundary Harnack inequality. Our contribution in this direction is the same as for the Carleson estimate, i.e., we derive an estimate where the constant does not depend on the solution. The proof is based on a barrier function estimate and this requires the domain to be $C^{1,1}$-regular. 
\begin{theorem}
	\label{the boundary Harnack} Assume that $\Omega$ is a $C^{1,1}$-regular domain such that $0 \in \partial \Omega$. Let $u,v \in C(B_{4R} \cap \overline{\Omega})$, with $R \in (0,R_0/4]$, be two positive solutions of \cref{thePDE}. Let $A_R \in B_{R/2L} \cap \Omega $ be such that $d(A_R,\partial \Omega) > R/(4L^2)$ and assume that $v(A_R)= u(A_R)>0$ and $v = u = 0$ on $\partial \Omega \cap B_{4R}$. There exists a constant $C$, which is independent of $u,v$ and of the radius $R$, and numbers $\mu_0, \mu_1 \in [0,\infty]$ such that $\mu_0 \leq u(A_R) \leq \mu_1$, 
	\begin{equation*}
		\sup_{x \in B_{R/C} \cap \Omega} \frac{v(x)}{u(x)} \leq \frac{\mu_1}{\mu_0} \,, 
	\end{equation*}
	and 
	\begin{equation*}
		\int_{\mu_0}^{\mu_1} \frac{dt}{R^2 \phi(t/R)+t} \leq C\,. 
	\end{equation*}
\end{theorem}
In the homogeneous case \Cref{the boundary Harnack} reduces to the classical boundary Harnack inequality, i.e., the ratio $v/u$ is bounded by a uniform constant. If $\phi$ satisfies the Osgood conditions \cref{osgood2} and \cref{osgood1} then \Cref{the boundary Harnack} implies that the ratio $v/u$ is bounded. In the general case when $\phi$ does not satisfy \cref{osgood2} and \cref{osgood1} the ratio $v/u$ can be unbounded. Note that we allow $\mu_0 = 0$ and $\mu_1 = \infty$. In this case, arguing as in the case of \Cref{mainthm}, we may still conclude that if $u(A_R)=1$ then the ratio $v/u$ is bounded when the radius $R$ is small enough. At the end of the paper we give an example which shows that in the model case $\phi(t)= (|\log t|+1)t$ \Cref{the boundary Harnack} is essentially sharp. 

\subsection{Consequences for the theory of the $p(x)$-Laplacian} Consider the $p(x)$-Laplace equation 
\begin{equation}
	\label{px}- \text{div} (|\nabla u|^{p(x)-2} \nabla u) = 0, \quad 1 < p(x) < \infty\,. 
\end{equation}
Let us make the assumption that $p(\cdot)$ is continuously differentiable. In non-divergence form this equation is of the form \cref{thePDE} and has a drift term which satisfies \cref{non-homogeneity} with $\phi(t)= C(|\log t|+1)t$ (see \cite{Vesku, JLP}). Solutions of \cref{px} are called $p(x)$-harmonic functions. 

Let us return to the previous outline of the proof of the Carleson estimate. In \cite{Alk, HKL, Wo} it was proved that a non-negative $p(x)$-harmonic function $u$ in $B_{2R}$ satisfies the following Harnack type estimate 
\begin{equation*}
	\sup_{B_R} u \leq C (\inf_{B_R}u+R) 
\end{equation*}
for a constant $C$ depending on the solution $u$. In \cite{AdLu}, Adamowicz and Lundstr\"om used the above estimate to prove a version of \cref{lapcarl} with a constant depending on the solution. From our perspective \Cref{mainthm} provides an improvement over this. Specifically, calculating the integral in \Cref{mainthm} in the context of the $p(x)$-Laplacian we obtain the following corollary. 
\begin{corollary}
	Let $\Omega$ be as in \Cref{mainthm} and $p \in C^1(\rn)$ such that $1 < p_- \leq p(x) \leq p_+ < \infty$. Assume that $u \in C(B_{4R}\cap \overline{\Omega})$, with $R \in (0,R_0/4]$, is a non-negative $p(x)$-harmonic function. Let $A_R \in B_{R/2L} \cap \Omega $ be a point such that $d(A_R,\partial \Omega) > R/(4L^2)$, and assume that $u = 0$ on $\partial \Omega \cap B_{4 R}$. There exists a constant $C(N,p_-,p_+,\|p\|_{C^1},L) > 1$ which is independent of $u$ and $R$ such that 
	\begin{equation}
		\nonumber \sup_{B_{R/C} \cap \Omega} u \leq C \max \left \{ u(A_R)^{1+C R},u(A_R)^{\frac{1}{1+CR}} \right \}\,. 
	\end{equation}
\end{corollary}

Let us now turn our attention to \Cref{the boundary Harnack}. An immediate corollary for $p(x)$-harmonic functions is. 
\begin{corollary}
	\label{p(x) bHp} Assume that $\Omega$ is $C^{1,1}$-regular domain such that $0 \in \partial \Omega$. Let $u,v \in C(B_{4R} \cap \overline{\Omega})$, with $R \in (0,R_0/4]$, be two positive $p(x)$-harmonic functions. Let $A_R \in B_{R/2L} \cap \Omega $ be such that $d(A_R,\partial \Omega) > R/(4L^2)$, and assume that $v(A_R)= u(A_R)>0$ and $v = u = 0$ on $\partial \Omega \cap B_{4R}$. There exists a constant $C(N,p_-,p_+,\|p\|_{C^1},L) > 1$ which is independent of $u,v$ and $R$ such that
	\[ \sup_{x \in B_{R/C} \cap \Omega} \frac{v(x)}{u(x)} \leq C \max \left \{ u(A_R)^{CR}, u(A_R)^{-CR} \right \} . \]
\end{corollary}
The above corollary is similar to the boundary Harnack inequality proved in \cite{AdLu}, but in \Cref{p(x) bHp} the constants does not depend on the solution. As we already mentioned we provide an example in \Cref{secbhi} that shows that \Cref{p(x) bHp} is essentially sharp.

\subsection{Organization of the paper}

In \Cref{secprel} we list all the assumptions on $\phi$ in \cref{non-homogeneity} and recall the definition of a Reifenberg flat domain. In \Cref{seccontblow} we prove the sharp H\"older continuity estimate up to the boundary in Reifenberg flat domains (\Cref{bdry holder cont}). By this we mean that we give the sharp H\"older norm in terms of the maximum of the solution. In \Cref{ssecblow} we study the blow-up rate of a solution near the boundary in NTA-domains (\Cref{blow up 2}). These results are crucial in the proof of the Carleson estimate but are of independent interest. In \Cref{secproof} we give the proof of the Carleson estimate (\Cref{mainthm}). In \Cref{secbhi} we prove the boundary Harnack estimate (\Cref{the boundary Harnack}). 

\section*{Acknowledgment} The first author was supported by the Swedish Research Council, dnr: 637-2014-6822. The second author was supported by the Academy of Finland grant 268393. 

\section{Preliminaries} \label{secprel}

Throughout the paper $B(x,r)$ denotes the open ball centered at $x$ with radius $r$. When the ball is centered at the origin we simply write $B_r$. Given a point $x \in \rn$ and a set $E \subset \rn$ we denote their distance by $d(x,E):= \inf_{y \in E}|x-y|$.

We recall the definition of a viscosity solution. 
\begin{definition}
	\label{visco_def} We call a lower semicontinuous function $u: \Omega \to \er$ a \emph{viscosity supersolution} of \cref{thePDE} in $\Omega$ if the following holds: if $x_ 0 \in \Omega$ and $\varphi \in C^2(\Omega)$ are such that $u- \varphi$ has a local minimum at $x_0$ then
	\[ F(D^2\varphi(x_0), D\varphi(x_0),x_0) \geq 0. \]
	An upper semicontinuous function $u: \Omega \to \er$ is a viscosity subsolution of \cref{thePDE} in $\Omega$ if the following holds: if $x_ 0 \in \Omega$ and $\varphi \in C^2(\Omega)$ are such that $u- \varphi$ has a local maximum at $x_0$ then
	\[ F(D^2\varphi(x_0), D\varphi(x_0),x_0) \leq 0. \]
	Finally a continuous function is a viscosity solution if it is both a super- and a subsolution. 
\end{definition}

As mentioned in the introduction we assume that $F$ in \cref{thePDE} has a drift term which satisfies the growth condition 
\begin{equation*}
	|F(0, p, x) \big| \leq \phi(|p|) 
\end{equation*}
for every $(x,p) \in \rn \times \rn$, where $\phi: [0,\infty) \to [0,\infty)$ is continuous, increasing, and satisfies the following structural conditions from \cite{Vesku}. For $t > 0$ we write $\phi(t)$ as 
\begin{equation}
	\nonumber \phi(t) = \eta(t)\, t 
\end{equation}
and assume the following. 
\begin{enumerate}
	[label=(P\arabic*)] 
	\item \label{P1} $\phi: [0, \infty) \to [0, \infty)$ is locally Lipschitz continuous in $(0, \infty)$ and $\phi(t)\geq t$ for every $t \geq 0$. Moreover, $\eta: (0, \infty) \to [1,\infty)$ is non-increasing on $(0,1)$ and nondecreasing on $ [1, \infty) $. 
	\item \label{P2} $\eta$ satisfies 
	\begin{equation*}
		\lim_{t \to \infty} \frac{t \eta'(t)}{\eta(t)} \log(\eta(t)) = 0. 
	\end{equation*}
	\item \label{P3} There is a constant $\Lambda_0$ such that 
	\begin{equation*}
		\eta(st) \leq \Lambda_0 \eta(s) \eta(t) 
	\end{equation*}
	for every $s,t \in (0, \infty)$. 
\end{enumerate}
The assumption \cref{P2} implies that $\eta$ is a slowly increasing function \cite{BGT}. We will repeatedly use the fact that for every $\eps>0$ there is a constant $C_\eps$ such that $\eta(t) \leq C_\eps t^{\eps}$ for every $t \geq 1$, see again \cite{BGT}. We explicitly note that our assumptions \cref{P1,P2,P3} do not rule out the possibility that $\phi(0)>0$, that $\phi$ is non-Lipschitz at $0$, and that the maximum/comparison-principle does not hold. Moreover the assumptions \cref{P1,P2,P3} do not imply that $\phi$ satisfies the Osgood conditions \cref{osgood2} and \cref{osgood1}.

We may replace the equation \cref{thePDE} by two inequalities which follow from the ellipticity condition and the modulus of continuity of the drift term \cref{non-homogeneity}. In other words if $u$ is a solution of \cref{thePDE} then it is a viscosity supersolution of 
\begin{equation}
	\label{model1} \mathcal{P}_{\lambda, \Lambda}^+(D^2 u) = -\phi(|Du|) 
\end{equation}
and a viscosity subsolution of 
\begin{equation}
	\label{model2} \mathcal{P}_{\lambda, \Lambda}^-(D^2 u) = \phi(|Du|) 
\end{equation}
in $\Omega$. Here $\mathcal{P}_{\lambda, \Lambda}^-, \mathcal{P}_{\lambda, \Lambda}^+$ are the usual Pucci operators, which are defined for a symmetric matrix $X \in \mathbb{S}^{n \times n}$ with eigenvalues $e_1, e_2, \dots, e_n$ as
\[ \mathcal{P}_{\lambda, \Lambda}^+(X):= -\lambda \sum_{e_i \geq 0} e_i - \Lambda \sum_{e_i<0} e_i \qquad \text{and} \qquad \mathcal{P}_{\lambda, \Lambda}^-(X):= -\Lambda \sum_{e_i \geq 0} e_i - \lambda \sum_{e_i<0} e_i . \]
We note that all the results of this paper hold if we instead of assuming that $u$ is a solution of \cref{thePDE} we only assume that it is a supersolution of \cref{model1} and a subsolution of \cref{model2}. We recall the result from \cite{Vesku}. 
\begin{theorem}
	\label{weakHarnack} Assume that $u\in C(B(x_0,2r))$, with $r \leq 1$, is a non-negative solution of \cref{thePDE}. Denote $m:= \inf_{B(x_0,r)}u$ and $M:= \sup_{B(x_0,r)}u$. There is a constant $C$ which is independent of $u$ and $r$ such that 
	\begin{equation*}
		\int_{m}^{M} \frac{dt}{r^2\phi(t/r)+ t} \leq C. 
	\end{equation*}
\end{theorem}

To describe the kind of domains we will be considering we first recall the definition of Reifenberg flat domains.
\begin{definition}
	\label{def:hyperplane:approximable} Let $ \Omega \subset \rn $ be a bounded domain. Then $\partial \Omega $ is said to be uniformly $ ( \delta, r_0 )$-approximable by hyperplanes, provided there exists, whenever $ w \in \partial \Omega$ and $ 0 < r < r_0, $ a hyperplane $ \pi $ containing $w$ such that 
	\begin{equation*}
		h ( \partial \Omega \cap B(w,r), \pi \cap B(w,r) ) \leq \delta r\,. 
	\end{equation*}
	Here $h ( E, F ) = \max ( \sup \{ d ( y, E ) : y \in F \}, \sup \{ d ( y, F ) : y \in E \} )$ is the Hausdorff distance between the sets $ E, F \subset \rn$ . 
\end{definition}
We denote by $ {\mathcal F} ( \delta, r_0 ) $ the class of all domains $ \Omega$ which satisfy \Cref{def:hyperplane:approximable}. Let $ \Omega \in {\mathcal F} ( \delta, r_0 )$, $w\in \partial\Omega$, $0<r<r_0,$ and let $\pi$ be as in \Cref{def:hyperplane:approximable}. We say that $\partial\Omega$ separates $ B_r(w), $ if 
\begin{equation}
	\label{eqn:separating} \{ x \in \Omega \cap B(w,r) : d ( x, \partial \Omega ) \geq 2 \delta r \} \subset \mbox{ one component of } \rn \setminus \pi. 
\end{equation}
\begin{definition}
	\label{def:riefenberg:flat} Let $ \Omega\subset\rn$ be a bounded domain. Then $ \Omega$ and $ \partial \Omega $ are said to be $ (\delta, r_0 )$-Reifenberg flat provided $ \Omega \in { \mathcal F} ( \delta, r_0 ) $, $\delta < 1/8$ and provided \cref{eqn:separating} holds whenever $ 0 < r < r_0, w \in \partial \Omega.$ 
\end{definition}
For short we say that $ \Omega $ and $\partial\Omega$ are $ \delta $-Reifenberg flat whenever $ \Omega $ and $\partial\Omega$ are $ ( \delta, r_0 ) $-Reifenberg flat for some $ r_0 > 0. $ We note that an equivalent definition of Reifenberg flat domains is given in \cite{KT}.

Next we recall the following definition of NTA-domains.
\begin{definition}
	\label{def:NTA} A bounded domain $\Omega$ is called non-tangentially accessible \textbf{(NTA)} if there exist $L \geq 2$ and $r_0$ such that the following are fulfilled: 
	\begin{enumerate}
		[label=(\roman*)] 
		\item \label{NTA1} \textbf{corkscrew condition:} for any $ w\in 
		\partial\Omega, 0<r<r_0,$ there exists a point $a_r(w) \in \Omega $ such that 
		\begin{equation}
			\nonumber L^{-1}r<|a_r(w)-w|<r, \quad d(a_r(w), 
			\partial\Omega)>L^{-1}r, 
		\end{equation}
		\item \label{NTA2} $\rn \setminus \Omega$ satisfies \cref{NTA1}, 
		\item \label{NTA3} \textbf{uniform condition:} if $ w \in 
		\partial \Omega, 0 < r < r_0, $ and $ w_1, w_2 \in B ( w, r) \cap \Omega, $ then there exists a rectifiable curve $ \gamma: [0, 1] \to \Omega $ with $ \gamma ( 0 ) = w_1,\, \gamma ( 1 ) = w_2, $ such that 
		\begin{enumerate}
			\item $H^1 ( \gamma ) \, \leq \, L \, | w_1 - w_2 |,$ 
			\item $\min\{H^1(\gamma([0,t])), \, H^1(\gamma([t,1]))\, \}\, \leq \, L \, d ( \gamma(t), 
			\partial \Omega)$, for all $t \in [0,1]$. 
		\end{enumerate}
	\end{enumerate}
\end{definition}

In \Cref{def:NTA}, $ H^1 $ denotes length or the one-dimensional Hausdorff measure. We note that \cref{NTA3} is different but equivalent to the usual Harnack chain condition given in \cite{JK} (see \cite{BL}, Lemma 2.5). Moreover, using \cite[Theorem 3.1]{KT} we see that there exists $\hat\delta=\hat\delta(N)>0$ such that if $\Omega\subset\rn$ is a $(\delta,r_0)$-Reifenberg flat domain and if $0<\delta\leq\hat\delta$, then $\Omega$ is an NTA-domain in the sense described above with constant $L=L(N)$. In the following we assume $0<\delta\leq\hat\delta$ and we refer to $L$ as the NTA constant of $\Omega$. 
\begin{remark}
	\label{SmallLip} Let $\Omega$ be a Lipschitz domain with constant $ l < 1/8$ then $\Omega$ is $\delta$-Reifenberg flat with constant 
	\begin{equation}
		\nonumber \delta = \sin(\arctan(l)) = \frac{l}{\sqrt{l^2+1}}\,. 
	\end{equation}
	Moreover note that $\delta < l$ and that any Lipschitz domain is also an NTA-domain. 
\end{remark}

\subsection{Reduction argument} \label{ssec reduction}

\subsubsection*{Reduction to small Lipschitz constant} First we observe that we may assume that the domain $\Omega$ in \Cref{mainthm} is Reifenberg flat with small $\delta$. Indeed assume $\Omega$ is an $l$-Lipschitz domain, and the equation \cref{thePDE} has ellipticity constants $\lambda$ and $\Lambda$. We may stretch the domain by a linear map $\mathcal{T}$ such that $\Omega'= \mathcal{T}(\Omega)$ is an $\hat{l}$-Lipschitz domain with $\hat{l} < \frac{1}{100}$. Moreover, if $u$ is a solution of \cref{thePDE} in $\Omega$ then $v(x) = u(\mathcal{T}^{-1}(x))$ is a solution of a similar equation with ellipticity constants $\tilde{\lambda}$ and $\tilde{\Lambda}$. Thus we may consider the case when $\Omega$ an $l$-Lipschitz domain, with $l \leq 1/100$. In particular, by \Cref{SmallLip} we may assume that $\Omega$ is Reifenberg flat with constant $1/100$.

\subsubsection*{Reduction to a canonical scale} In the proof of \Cref{mainthm} we prefer to scale the radius $R$ to one and $R_0 = 16$. In this way we do not get confused by the many radii which appear in the proof. Let us assume that $u$ is as in the theorem. By rescaling $u_R(x) := \frac{u(Rx)}{R}$ we obtain a function $u_R$ which is a solution to the equation 
\begin{equation}
	\label{rescaledPDE} F_R(D^2 u_R, D u_R,x)=0 
\end{equation}
where 
\newcommand{\phiRa}{\phi_R} 
\newcommand{\phiRb}{\Phi_R} 
\begin{equation}
	\label{rescaled-non-homogeneity} | F_R(0,p,\cdot)| \leq \phiRa(|p|) := R \phi(|p|). 
\end{equation}
Note that \cref{rescaledPDE} is of type \cref{thePDE}, satisfying \cref{ellipticity} with the same constants and with nonlinearity $\phiRa$. 

Since $\phiRa$ does not satisfy \cref{P1,P2,P3} we need to rephrase \Cref{weakHarnack} in our new scale as we cannot use it directly for $u_R$ (see \Cref{corHarnack}). With this in mind we prove the H\"older regularity estimates (\Cref{holder cont} and \Cref{bdry holder cont}) and the blow-up estimate (\Cref{blow up 2}) assuming that we have a solution of \cref{rescaledPDE,rescaled-non-homogeneity}. 

The simplifying point is that if we denote 
\begin{equation}
	\label{phi R} \phiRb(t) := \phiRa(t) + t \geq \phiRa\,, 
\end{equation}
then we see that $\phiRb$ satisfies \cref{P1,P2} with $\eta_R(t) = R\eta(t)+1$ and instead of \cref{P3} it satisfies 
\begin{enumerate}
	[label=(P3')] 
	\item \label{P3'} 
	\begin{equation*}
		\eta_R(st) \leq \Lambda_0 \eta(s) \eta_R(t), \qquad \text{for every $s,t \in (0, \infty)$,} 
	\end{equation*}
\end{enumerate}
with the $\Lambda_0$ from \cref{P3} for $\phi$. Rephrasing \Cref{mainthm} in terms of $u_R$ we see that if we denote $M_R = \sup_{B_1 \cap \Omega} u_R$ and $m_R = \frac{u(A_R)}{R}$, then \Cref{mainthm} becomes 
\begin{equation}
	\label{mainthmrescaled} \int_{m_R}^{M_R} \frac{dt}{\phiRb(t)} \leq C \,. 
\end{equation}
Thus our aim will be to prove that for a solution of \cref{rescaledPDE,rescaled-non-homogeneity}, \cref{mainthmrescaled} holds.
\begin{corollary}
	\label{corHarnack} Assume that $u\in C(B(x_0,2r))$, is a non-negative solution of \cref{rescaledPDE,rescaled-non-homogeneity}. Denote $m:= \inf_{B(x_0,r)}u$ and $M:= \sup_{B(x_0,r)}u$. Let $\alpha_0 \in (0,1)$, then there exists a constant $C(\alpha_0) > 1$ which is independent of $u$, $r$ and $R$ such that 
	\begin{equation*}
		\int_{m}^{M} \frac{dt}{ r^{\alpha} \phiRb(t)+ t} \leq C, \quad \forall \alpha \in [0,\alpha_0]\,. 
	\end{equation*}
\end{corollary}
\begin{proof}
	We define $v \in C(B(x_0,2\rho))$, where $\rho = rR \leq 1$, by rescaling $v(y) = Ru(y/R)$. Then $v$ is a solution of \cref{thePDE} with non-homogeneity $\phi$ and \Cref{weakHarnack} implies 
	\begin{equation}
		\label{scaling weak harnack} C \geq \int_{Rm}^{RM} \frac{ds}{\rho^2\phi(s/\rho)+ s} = \int_{m}^{M} \frac{dt}{R r^2\phi(t/r)+ t}. 
	\end{equation}
	Since $\eta$ is slowly increasing function we have $\eta(t) \leq C_\ep t^{\ep}$ for all $t>1$ and for any $\ep$. It is now easy to see that if $\eps \geq \eps_0$ for some fixed $\eps_0 \in (0,1)$ then $\eta(t) \leq C_{\ep_0} t^\ep$. Therefore by \cref{P3} we deduce that for any $\alpha \in [0,\alpha_0]$ we have 
	\begin{equation*}
		Rr^2\phi(t/r) = r^2 R\eta(t/r) \frac{t}{r} \leq \Lambda_0 r \, \eta(1/r) R\eta(t)t \leq \Lambda_0 C_{1-\alpha_0} r r^{\alpha-1} \, R\phi(t) \leq C r^{\alpha}\, \phiRb(t)\,, 
	\end{equation*}
	for a constant $C(\alpha_0) > 1$. Plugging this into \cref{scaling weak harnack} gives the result. 
\end{proof}

\section{H\"older continuity estimates} \label{seccontblow}

In this section we prove interior and boundary H\"older continuity estimates (\Cref{holder cont} and \Cref{bdry holder cont}) when $\Omega$ is Reifenberg flat. We note that solutions of \cref{thePDE} are known to be H\"older continuous \cite{Si}. The point of the following results is to derive the sharp H\"older norm with respect to the $L^\infty$-norm of the solution. As we mentioned in the previous section we assume that $u$ solution of \cref{rescaledPDE,rescaled-non-homogeneity}.
\begin{theorem}
	\label{holder cont} Let $u\in C(B(x_0, 2r))$, with $r \leq 1$, be a solution of \cref{rescaledPDE}, \cref{rescaled-non-homogeneity}, and denote $M = \sup_{B(x_0,r)}|u|$. Then $u$ is H\"older continuous, i.e. for every $\rho \leq r$ the following holds 
	\begin{equation*}
		\osc_{B(x_0,\rho)} u \leq C_1 M \left(\frac{\rho}{r} \right)^\alpha + C_1\phiRb(M) \rho^{1/4} r^{1/4} 
	\end{equation*}
	for some $C_1$ and $\alpha \in (0,1/4)$, which are independent of $u, r$ and $R$. We also have the following oscillation decay, there exist a constant $C=C(\Lambda/\lambda,N)>0$ and $\tau = \tau(\Lambda/\lambda,N) \in (0,1)$ such that 
	\begin{equation}
		\label{oscdec} \osc_{B(x_0,\rho/2)} u \leq \tau \osc_{B(x_0,\rho)} u + C \phiRb(M) \sqrt{\rho}, \quad \rho \in (0,r). 
	\end{equation}
\end{theorem}
\begin{proof}
	Without loss of generality we may assume that $x_0 = 0$. For every $\rho \leq r$ we denote $M_{\rho} = \sup_{B_\rho} u$ and $m_{\rho} = \inf_{B_\rho} u$. Define functions $v(x) = M_{\rho}-u$ and $w(x)= u(x) - m_{\rho}$ which are non-negative in $B_\rho$. Denote $M_{v,\rho/2} = \sup_{B_\rho/2} v$, $m_{v,\rho/2} = \inf_{B_\rho/2} v$ and $M_{w,\rho/2}$ and $m_{w,\rho/2}$ for the supremum and infimum of $w$. Note that $M_{v,\rho/2}, M_{w,\rho/2} \leq 2M$. Since $v$ is a solution of \cref{rescaledPDE,rescaled-non-homogeneity} we obtain from \Cref{corHarnack} with $\alpha = 1/2$ that 
	\begin{equation*}
		\int_{m_{v,\rho/2}}^{M_{v,\rho/2}} \frac{dt}{ \sqrt{\rho} \phiRb(M) + t} \leq C. 
	\end{equation*}
	By integrating this we deduce 
	\begin{equation*}
		M_{v,\rho/2} \leq C m_{v,\rho/2} + C \phiRb(M)\sqrt{\rho}. 
	\end{equation*}
	This in turn implies 
	\begin{equation}
		\label{from harnack v} M_\rho - m_{\rho/2} \leq C (M_\rho - M_{\rho/2}) + C \phiRb(M)\sqrt{\rho}. 
	\end{equation}
	Similar argument applied to $w$ yields 
	\begin{equation}
		\label{from harnack w} M_{\rho/2} - m_{\rho} \leq C ( m_{\rho/2} - m_{\rho}) + C \phiRb(M)\sqrt{\rho}. 
	\end{equation}
	
	Denote $\omega(\rho) = \osc_{B_\rho} u$. Adding \cref{from harnack v} and \cref{from harnack w} gives 
	\begin{equation}
		\nonumber \omega(\rho/2) \leq \tau \omega(\rho) + C \phiRb(M) \sqrt{\rho} 
	\end{equation}
	for every $\rho \leq r \leq 1$ where $\tau = \frac{C-1}{C+1}<1$. This is \cref{oscdec}. Moreover by \cite[Lemma 8.23]{GT} the following holds 
	\begin{equation*}
		\omega(\rho) \leq C\omega(r) \left( \frac{\rho}{r}\right)^{\alpha} + C \phiRb(M) \rho^{1/4} r^{1/4} 
	\end{equation*}
	for some $\alpha >0$. 
\end{proof}

We will turn our attention to the H\"older continuity up to the boundary for solutions vanishing at the boundary. The boundary regularity does not follow directly from the interior regularity. There is an additional difficulty due to the fact that the comparison principle does not hold for \cref{thePDE}. In fact, even the maximum principle in general is not true. We need two lemmas which allow us to overcome the lack of comparison principle. 

The first lemma shows the existence of the maximal solution of the equation \cref{model2} with a given Dirichlet boundary data. Here we do not need the assumption \cref{P3} so we may state the result for $\phi$ instead of $\phiRb$. 
\begin{lemma}
	\label{maximal comparison function} Let $\Omega$ be a bounded Lipschitz domain. Assume that $u \in C(\overline{\Omega})$ is a subsolution of \cref{model2} in $\Omega$. Then there exists a function $v \in C(\overline{\Omega})$ which is a solution of \cref{model2} in $\Omega$ such that $ u=v$ on $
	\partial \Omega$ and $u \leq v$ in $ \Omega$. 
\end{lemma}
\begin{remark}
	\label{maximal comparison function remark} Assume that we have a subsolution of \cref{model2} with nonlinearity $\phiRa$. Then by scaling to the original scale as in the proof of \Cref{corHarnack} we get a subsolution to \cref{model2} with nonlinearity $\phi$, consequently we can apply \Cref{maximal comparison function} and subsequently scale back to the canonical scale to obtain that \Cref{maximal comparison function} also holds for subsolutions of \cref{model2} with nonlinearity $\phiRa$. 
\end{remark}
\begin{proof}
	For $\eps>0$ small we define 
	\begin{equation*}
		\phi_\eps(t):= (1+\eps)\max\{\phi(t), \phi(\eps) \}, 
	\end{equation*}
	where $\phi$ is from \cref{non-homogeneity}. It is straightforward to check that $\phi_{\eps}$ satisfies the assumptions \cref{P1,P2,P3}. We claim that there exists a solution $v_\eps \in C(\overline{\Omega})$ of 
	\begin{equation}
		\label{bar eq} 
		\begin{cases}
			&\mathcal{P}^-_{\lambda, \Lambda}(D^2v) = \phi_\eps(|Dv|) \quad \text{in }\, \Omega,\\
			&v=u \qquad \text{on }\, 
			\partial \Omega. 
		\end{cases}
	\end{equation}
	
	The existence of $v_\eps$ follows from \cite{Si}. We need to check that \cref{bar eq} satisfies the assumptions in \cite{Si}. First we write \cref{bar eq} as
	\[ \mathcal{P}^-_{\lambda, \Lambda}(D^2v) - \phi_\eps(|Dv|) + (1+\eps) \phi(\eps) = (1+\eps)\phi(\eps) \]
	and denote $F(X,p) = \mathcal{P}^-_{\lambda, \Lambda}(X) - \phi_\eps(|p|) + (1+\eps)\phi(\eps)$. Then it holds that
	\[ F(0,0) = 0. \]
	Since the Pucci operator is uniformly elliptic \cite{CC} we only have to check that 
	\begin{equation}
		\label{check sira} \phi_\eps(t)- \phi_\eps(s) \leq C_1(t+s)|t-s| + C_2|t-s| 
	\end{equation}
	holds for every $s,t \geq 0$. Note that we allow the constants above to depend on $\eps$. Since $\phi$ is locally Lipschitz and satisfies the condition \cref{P2} we have for every $t \geq \eps$ that
	\[ \phi'(t) = \eta'(t)t+ \eta(t) \leq C \eta(t) \leq C t, \]
	where the last inequality follows from the $\eta(t) \leq Ct$ for every $t \geq 1$. Thus we have
	\[ \phi_\eps(t)- \phi_\eps(s) \leq \sup_{\eps \leq \xi \leq t+s+ \eps} \phi'(\xi)\, |t-s| \leq C\sup_{\xi \leq t+s+ \eps} |\xi| |t- s| \leq C(t+s + \eps)|t-s|. \]
	Hence we have \cref{check sira} and the existence $v_\epsilon$ follows.
	
	Let $0< \eps_1 < \eps_2$ and let $v_{\eps_1}$ and $v_{\eps_2}$ be solutions of the corresponding equations \cref{bar eq}. Let us show that the solutions are monotone with respect to $\eps$, i.e., 
	\begin{equation}
		\label{comparison claim} v_{\eps_1}(x) \leq v_{\eps_2}(x) \qquad \text{for every }\, x \in \Omega. 
	\end{equation}
	
	The claim \cref{comparison claim} follows from the standard comparison principle for semicontinuous functions and we only give the sketch of the argument. For more details and for the notation see \cite[Section 3]{CIL}. Assume that the claim does not hold. Then we conclude that there exists points $x,y \in \Omega$, a vector $p \in \rn$ and symmetric matrices $X,Y$ such that $X \leq Y$ and the pair $(p,X)$ belongs to the semi-jet $\bar{D}^{2,+}v_{\eps_1}(x)$ and $(p,Y)$ belongs to the semi-jet $\bar{D}^{2,-}v_{\eps_2}(y)$. If the functions $v_{\eps_1}, v_{\eps_2}$ were $C^2$ regular this would mean that $p = Dv_{\eps_1}(x) = Dv_{\eps_2}(y)$, $X = D^2v_{\eps_1}(x)$ and $Y = D^2v_{\eps_2}(y)$. Since $v_{\eps_1}$ is a subsolution of \cref{bar eq} we have
	\[ \mathcal{P}^-_{\lambda, \Lambda}(X) - \phi_{\eps_1}(|q|) \leq 0 \]
	and since $v_{\eps_2}$ is a supersolution of \cref{bar eq} we have
	\[ \mathcal{P}^-_{\lambda, \Lambda}(Y) - \phi_{\eps_2}(|q|) \geq 0. \]
	On the other hand, it follows from $X \leq Y$ and $\phi_{\eps_1}(|q|)< \phi_{\eps_2}(|q|)$ that
	\[ 0 \geq \mathcal{P}^-_{\lambda, \Lambda}(X) - \phi_{\eps_1}(|q|) >\mathcal{P}^-_{\lambda, \Lambda}(Y) - \phi_{\eps_2}(|q|) \geq 0 \]
	which is a contradiction. Thus we have \cref{comparison claim}. 
	
	We note that repeating the above argument we get that $u(x) \leq v_{\eps}(x)$ for every $x \in \Omega$ and for every $\eps>0$. Hence we have that $v_\eps(x) \to v(x)$ point-wise in $\Omega$ and $v \geq u$. It follows from the interior H\"older regularity \Cref{holder cont} that $v_\eps$ are locally uniformly H\"older continuous in $\Omega$. Therefore $v$ is continuous in $\Omega$ and by a standard viscosity convergence argument it is a solution of \cref{model2} in $\Omega$. Moreover, it follows from \cref{comparison claim} that for every $\eps\in (0, \eps_0)$ the following holds
	\[ u(x) \leq v_{\eps}(x) \leq v_{\eps_0}(x) \qquad \text{for every }\, x \in \overline{\Omega}. \]
	Since $u= v_{\eps_0}$ on $
	\partial \Omega$ we conclude that $v \in C(\overline{\Omega})$ and $v =u$ on $
	\partial \Omega$. 
\end{proof}

The next result shows that in small balls the equation \cref{model2} almost satisfies the maximum principle. 
\begin{lemma}
	\label{almost maximum princ} Let $v \in C(\overline{B}_r)$ be a subsolution of \cref{model2} in $B_r$ with non-homogeneity $\phiRb$ and $r\leq 1$ such that $v \leq M$ on $
	\partial B_r$ and let $\sigma>1$. There is a constant $c_0>0$, which depends on $\sigma$, such that if $r \leq \frac{c_0}{\eta_R(M)^2}$ then
	\[ \sup_{B_r}v \leq \sigma M. \]
	Furthermore if \cref{osgood1} holds then the maximum principle holds, i.e.,
	\[ \sup_{B_r}v \leq M. \]
\end{lemma}
\begin{proof}
	For every small $\eps>0$ we define
	\[ \phi_\eps(t):= (1+\eps)\max\{\phiRb(t), \phiRb(\eps) \}, \]
	We first let $\varepsilon_0 = 1/2$ and consider $r_0$ so small that 
	\begin{equation}
		\nonumber \lambda \int_0^1 \frac{ds}{\phi_{\varepsilon_0}(s)} > 2r_0, 
	\end{equation}
	where $\lambda>0$ is the ellipticity constant of the Pucci operator $\mathcal{P}^-_{\lambda,\Lambda}$. Clearly we may choose $r_0$ such that it does not depend on $R\leq1$. Note that this implies 
	\begin{equation}
		\nonumber \lambda \int_0^1 \frac{ds}{\phi_{\varepsilon}(s)} > 2r, \qquad \text{for all } \, \varepsilon \in [0,\varepsilon_0],\, r \in (0,r_0]\,. 
	\end{equation}
	When $0 < r \leq r_0$ we may define a function $f_\eps:[0,r] \to [0,\infty)$ by the implicit function theorem as
	\[ t = \lambda \int_0^{f_\eps(t)} \frac{ds}{\phi_\eps(s)} \]
	In particular, we have $f_{\eps} < 1$ by the choice of $r_0$. Define a function $g_\eps:[0,r] \to [0,\infty)$ as
	\[ g_\eps(t) := \int_0^t f_\eps(s)\, ds. \]
	Then $g_\eps$ is increasing and satisfies $g_\eps'' =\lambda^{-1} \phi_{\eps}(g_\eps') $. We define a radial function $w_\eps: \overline{B}_r \to \er$ as
	\[ w_\eps(x) =g_\eps(r) + M - g_{\eps}(|x|). \]
	Now note that $\phi_\eps(s)$ is constant for $s \leq \eps$ and thus $g_\eps(s) = c s^2$ close to $0$. It is therefore straightforward to see that $w_\eps \in C^2(B_r)$ and a calculation shows that
	\[ \mathcal{P}^-_{\lambda, \Lambda}(D^2w) \geq \phi_\eps(|Dw|) \qquad \text{in }\, B_r \]
	and $w_\eps = M \geq v$ on $
	\partial B_r$. It follows from the fact that $v$ is a subsolution of \cref{model2} with $\phiRb$ instead of $\phi$, and from $\phi_\eps > \phiRb$ that $v - w_\eps$ does not attain local maximum in $B_r$ (see the proof of \Cref{maximal comparison function}). Hence we have $w_\eps \geq v$ in $B_r$. 
	
	Now if $\phi$ satisfies the Osgood condition 
	\begin{equation}
		\label{osgo} \int_0^{1} \frac{ds}{\phi(s)} = \infty 
	\end{equation}
	then we easily see that $w_\eps \to M$ and we get the maximum principle, i.e., 
	\begin{equation}
		\nonumber \sup_{B_r}w \leq M. 
	\end{equation}
	
	With the above in mind let us assume that \cref{osgo} does not hold. In this case it is easy to see that $f_\eps \to f$ uniformly and for every $t \in (0,r)$ it holds that 
	\begin{equation}
		\label{def of f} t = \lambda \int_0^{f(t)} \frac{ds}{\phiRb(s)}. 
	\end{equation}
	Moreover $g_\eps \to g$ uniformly with $g(t) = \int_0^t f(s)\, ds$ and $w_\eps \to w$ with 
	\begin{equation}
		\nonumber w(x) = g(r) + M - g(|x|). 
	\end{equation}
	We still have $w \geq v$ in $B_r$ and $w = M$ on $
	\partial B_r$. Thus it is enough to show that
	\[ \sup_{B_r}w \leq \sigma M \,, \]
	which by the definition of $w$ is equivalent to 
	\begin{equation*}
		g(r) \leq (\sigma-1) M =: \mu M. 
	\end{equation*}
	
	We argue by contradiction and assume $g(r) > \mu M$. Since $g(0)= 0$ it follows from the mean value theorem that there exists $\xi < r$ such that $g'(\xi) = f(\xi)= \frac{\mu M}{r}$. Therefore it follows from \cref{def of f} that 
	\begin{equation}
		\label{almost max long} 
		\begin{split}
			r > \xi = \lambda \int_0^{f(\xi)} \frac{ds}{\phiRb(s)} &= \lambda \int_0^{ \frac{\mu M}{r}} \frac{ds}{\phiRb(s)} \geq \lambda \int_{\frac{\mu M}{2r}}^{\frac{\mu M}{r}} \frac{ds}{\eta_R(s) s} \\
			&\geq \lambda \left( \sup_{\frac{\mu M}{2r} <t < \frac{\mu M}{r}} \eta_R(t) \right)^{-1} \int_{\frac{\mu M}{2r}}^{\frac{\mu M}{r}} \frac{ds}{s}\\
			&\geq \log 2 \, \lambda \left( \sup_{\frac{\mu M}{2r} <t < \frac{\mu M}{r}} \eta_R(t) \right)^{-1}. 
		\end{split}
	\end{equation}
	By the assumption \cref{P3'} we obtain
	\[ \sup_{\frac{\mu M}{2r} <t < \frac{\mu M}{r}} \eta_R(t) \leq \Lambda_0^3 \eta(\mu) \eta(1/r) \eta_R(M) \leq C \eta_R(M) r^{-1/2}, \]
	where the last inequality follows from $\eta(t) \leq C\sqrt{t}$ for $t \geq 1$. Hence by \cref{almost max long} we have
	\[ \sqrt{r} >\frac{c}{\eta_R(M)}\,, \]
	which contradicts the assumption
	\[ r \leq \frac{c_0}{\eta_R(M)^2} \]
	when $c_0>0$ is small enough. Moreover since $\eta_R \geq 1$ we see that 
	\begin{equation}
		\nonumber \frac{c_0}{\eta_R(M)^2} \leq c_0 \leq r_0 
	\end{equation}
	if $c_0 > 0$ is again small enough. 
\end{proof}

We will now prove the boundary H\"older continuity. 
\begin{theorem}
	\label{bdry holder cont} Let $\Omega \subset \rn$ be a $\delta$-Reifenberg flat domain, and assume $x_0 \in 
	\partial \Omega$. Let $u \in C(B(x_0,r) \cap \overline \Omega )$, with $0 < r \leq 1$, be a non-negative solution of \cref{rescaledPDE,rescaled-non-homogeneity} such that $u=0$ on $
	\partial \Omega$ and denote $M = \sup_{B(x_0,r) \cap \Omega} u$. Then if $\delta \leq \frac{1}{100}$, $u$ is H\"older continuous and for every $0 < \rho \leq r$ the following holds 
	\begin{equation*}
		\sup_{B(x_0, \rho) \cap \Omega } u \leq C_1 M \left(\frac{\rho}{r} \right)^\alpha + C_1\phiRb(M) \rho^{2\alpha} r^{2\alpha} 
	\end{equation*}
	for some $C_1$ and $\alpha \in (0,1/8)$. 
\end{theorem}
\begin{proof}
	We may assume that $\delta = \frac{1}{100}$, since the proof does not improve with smaller value of $\delta$. We may also assume that $x_0 = 0$. In this proof we will proceed with a comparison construction which is more or less standard when dealing with $\delta$-Reifenberg domains. 
	
	First let $\tau \in (0,1)$ be from \cref{oscdec} and let $\kappa \in (0,1/4)$ be a number which we choose later. Let us fix $\sigma := \frac{\tau+1}{2 \tau}$ and let $c_0(\sigma)$ be the constant from \Cref{almost maximum princ}. 
	
	For the forthcoming iteration argument we wish to prove the following. Assume that we are given a radius $0 < \rho_0 \leq r$ and an upper bound $M_0$ which satisfy 
	\begin{equation}
		\label{osciassump} \sup_{B_{\rho_0} \cap \Omega} u \leq M_{0} \leq \frac{1}{\sigma} M \qquad \text{and}\qquad \rho_0 \leq \frac{c_0}{\eta_R(M_0)^2}. 
	\end{equation}
	Let us define the following quantities, 
	\begin{align*}
		\rho_1 &= \kappa \rho_0, \quad \tau_1 = \frac{\tau+1}{2}, \\
		M_1 &= \min\{\tau_1 M_0 + C \phiRb(M) \sqrt{\rho_0},M_0\}. 
	\end{align*}
	We will show that when $\kappa>0$ is chosen small enough then \cref{osciassump} implies 
	\begin{equation}
		\label{osci} \sup_{B_{\rho_1} \cap \Omega} u \leq M_{1} \qquad \text{and}\qquad \rho_1 \leq \frac{c_0}{\eta_R(M_1)^2}. 
	\end{equation}
	Moreover the choice of $\kappa$ is independent of $u, r$ and $R$.
	
	We begin by proving the first claim of \cref{osci}. To do this we will fix the scale $\rho_0$ in \Cref{def:hyperplane:approximable} and obtain a hyperplane $\pi$ and a set of local coordinates such that after rotation, the normal of the plane is in the $e_N$ direction and $\Omega \cap B_\rho \subset \{x_N > -2 \delta \rho_0 \}$. Let us choose a point $z_0 = (0,-2\delta \rho_0) \in \er^{N-1}\times \er$ and notice that 
	\begin{align}
		\label{ball inclusion} B_{\rho_1} \subset B(z_0, \rho_0/4) \subset B(z_0, \rho_0/2) \subset B_{\rho_0}. 
	\end{align}
	Denote $D := B(z_0, \rho_0/2) \cap \{x_N > -2\delta \rho_0\}$. First note that in the domain $D \cap \Omega$ $u$ is a subsolution of \cref{model2} and vanishes continuously on $D \cap 
	\partial \Omega$. It is easy to see that if we extend $u \equiv 0$ to $D \setminus \overline{\Omega}$ then $u$ is a subsolution of \cref{model2} in $D$ with nonlinearity $\phiRa$. By \Cref{maximal comparison function,maximal comparison function remark} there exists $v \in C(\overline{D})$ which is a solution of the following Dirichlet problem 
	\begin{equation}
		\label{Q1comparison} 
		\begin{cases}
			\mathcal{P}_{\lambda, \Lambda}^-(D^2 v) = \phiRa(|Dv|) &\text{ in } D \\
			v = 0 & \text{ on } 
			\partial D \setminus \overline{\Omega} \\
			v = u & \text{ on } 
			\partial D \cap \Omega 
		\end{cases}
	\end{equation}
	and $u \leq v$ in $D$. Since $v$ solves \cref{model2} we can (again by extending as zero) use \cref{osciassump} and \Cref{almost maximum princ} to get that
	\[ \sup_{D} v \leq \sigma M_0 =: \hat M_0\,. \]
	Moreover, since the equation in \cref{Q1comparison} is of the type \cref{rescaledPDE,rescaled-non-homogeneity} we can via reflection get a signed solution of an equation of the same type in $B(z_0, \rho_0/2) $ and therefore use the estimate \cref{oscdec} in \Cref{holder cont} to get 
	\begin{equation*}
		\sup_{ B(z_0, \rho_0/4) \cap \Omega} v \leq \tau \hat M_0 + C \phiRb(\hat M_0) \sqrt{\rho_0}. 
	\end{equation*}
	Since $u \leq v$ in $D$ we have by \cref{ball inclusion} and \cref{osciassump} that
	\[ \sup_{B_{\rho_1} \cap \Omega} u \leq M_{1} \]
	which implies the first claim of \cref{osci}.
	
	To prove the second claim of \cref{osci}, the aim is to choose $\kappa$ such that it holds. First note that $\tau M_0 \leq M_1 \leq M_0$, which can be rephrased as 
	\begin{equation}
		\nonumber \tau \leq \frac{M_1}{M_0} \leq 1\,. 
	\end{equation}
	Hence using \cref{P3'} and the fact that on $(0,1)$ $\eta$ is non-increasing, we get 
	\begin{equation}
		\nonumber \eta_R(M_1) \leq \Lambda_0 \eta(M_1/M_0) \eta_R(M_0) \leq \Lambda_0 \eta(\tau) \eta_R(M_0)\,. 
	\end{equation}
	From \cref{osciassump} together with the above we get 
	\begin{equation}
		\nonumber \frac{c_0}{\eta_R(M_1)^2} \geq \frac{c_0}{\Lambda_0^2 \eta(\tau)^2 \eta_R(M_0)^2} \geq \frac{c_0}{\Lambda_0^2 \eta(\tau)^2} \rho_0\,. 
	\end{equation}
	Thus by choosing $\kappa = \min\{\frac{c_0}{\Lambda_0^2 \eta(\tau)^2},1/8\}$ we have proved the second claim of \cref{osci}.
	
	Let us now choose $M_0 := \frac{1}{\sigma} M$ and a radius $\rho_0 \leq r$ such that 
	\begin{equation}
		\label{starting radii} \rho_0 := \min \bigg\{ \frac{c_0}{\eta_R(M_0)^2},r \bigg \}. 
	\end{equation}
	Define the radii 
	\begin{equation}
		\nonumber \rho_j := \kappa^j \rho_0 
	\end{equation}
	and the quantities 
	\begin{equation}
		\nonumber M_j := \min \big \{\tau_1 M_{j-1} + C \phiRb(M) \sqrt{\rho_{j-1}},M_j\big \}. 
	\end{equation}
	Iterating the implication \cref{osciassump} $\implies$ \cref{osci} we obtain that 
	\begin{equation}
		\label{oscdecay} \osc_{B_{\rho_j}} u \leq M_j, \quad j=1,\ldots. 
	\end{equation}
	Consider now the function $\omega(\rho):[0,\rho_0]$ such that $\omega(0)=0$ and 
	\begin{equation}
		\nonumber \omega(\rho_j) = M_j, \quad \omega(t \rho_{j+1} + (1-t) \rho_j) = M_{j+1} + t(M_{j}-M_{j+1}) \qquad \text{ for $t \in (0,1)$}. 
	\end{equation}
	Fix a radius $\rho \in (0,\rho_0]$. Let $k$ be such that $\rho \in [\rho_{k+1},\rho_k]$. Then $\kappa^2 \rho \in [\rho_{k+3},\rho_{k+2}]$. Let us estimate 
	\begin{equation}
		\label{iteration inequality} \omega(\kappa^2 \rho) \leq \omega(\rho_{k+2}) \leq \tau_1 \omega(\rho_{k+1}) + C \phiRb(M) \sqrt{\rho_{k+1}} \leq \tau_1 \omega(\rho) + C \phiRb(M) \sqrt{\rho}\,. 
	\end{equation}
	Thus as in \Cref{holder cont}, we may use \cref{iteration inequality} and \cite[Lemma 8.23 (with $\tau = \kappa^2,\, \mu=1/2,\, \gamma = \tau_1$)]{GT} to get an $\tilde \alpha \in (0,1/4)$ such 
	\begin{equation}
		\label{iteration estimate1} \omega(\rho) \leq C \left (\frac{\rho}{\rho_0} \right )^{\tilde \alpha}\omega(\rho_0) + C \phiRb(M) \rho^{ \tilde \alpha} \rho_0^{ \tilde \alpha}, \quad \rho \in [0,\rho_0]\,. 
	\end{equation}
	Using \cref{oscdecay}, \cref{iteration estimate1}, the definition of $\omega$ and arguing as before, we get for a new constant $C > 1$ 
	\begin{equation}
		\label{iteration estimate2} \osc_{B_\rho \cap \Omega} u \leq C \left (\frac{\rho}{\rho_0} \right )^{\tilde \alpha} M + C \phiRb(M) \rho^{\tilde \alpha} \rho_0^{\tilde \alpha}, \quad \rho \in [0,\rho_0]\,. 
	\end{equation}
	
	We have thus proved the oscillation decay for possibly small radii $\rho \leq \rho_0$. To finish the proof let us show that \cref{iteration estimate2} implies the result. First if $\rho_0 = r$ we are done. Let us assume that $\rho_0 < r$. If $\rho \in (\rho_0,r)$ then it follows from $\tilde \alpha \in (0,1/4)$ and \cref{starting radii} that $\rho_0^{2} \geq \frac{c}{\eta_R(M)}$ and therefore 
	\begin{equation}
		\label{iteration estimate3a} \osc_{B_\rho \cap \Omega} u \leq M \leq C(\rho_0^{2} \eta_R(M))M \leq C \phiRb(M) \rho^{\tilde \alpha} r^{\tilde \alpha} . 
	\end{equation}
	On the other hand if $\rho \leq \rho_0$ we get from Young's inequality that 
	\begin{equation*}
		\label{iteration estimate3} M \left( \frac{\rho}{\rho_0} \right )^{\tilde \alpha} \leq M \left ( \frac{\rho}{r} \right )^{ \tilde \alpha} + \rho^{ \tilde \alpha} r ^{ \tilde \alpha} \rho_0^{-2\tilde \alpha} M \leq M \left ( \frac{\rho}{r} \right )^{ \tilde \alpha} + C\phiRb(M) \rho^{\tilde \alpha} r ^{ \tilde \alpha} 
	\end{equation*}
	where the last inequality follows from $\rho_0^{2} \geq \frac{c}{\eta_R(M)}$. Therefore by the above inequality together with \cref{iteration estimate2} and by \cref{iteration estimate3a} we get 
	\begin{equation*}
		\osc_{B_\rho \cap \Omega} u \leq C \left (\frac{\rho}{r} \right )^{\tilde \alpha} M + C \phiRb(M) \rho^{\tilde \alpha} r^{\tilde \alpha}, \quad \rho \in [0,r]\,. 
	\end{equation*}
	Denoting $\alpha = \tilde \alpha/2$ finishes the proof. 
\end{proof}

\section{Blow up estimates} \label{ssecblow}

In this section we study the blow-up profile of non-negative solutions of \cref{thePDE} near the boundary of an NTA-domain. The goal is to prove that every solution grows first with a growth-rate prescribed by the non-homogeneity of the equation. Then we show that there exists a critical value, which we are able to control in a quantitative way, such that after the critical value the solution blow-up as the solutions of the homogeneous equation. 

We begin by describing this critical value of the solutions. As in \Cref{ssec reduction} we may rescale such that if we have an $(L,r_0)$-NTA-domain, with $L\geq 2,$ we may assume instead that we have an $(L,2 L^3)$-NTA-domain (the NTA constant $L\geq 2$ is independent of scaling), $\Omega \subset \rn$, and by translation we can assume that $0 \in\partial \Omega$. From \cite{Jones} (see also \cite[Theorem 3.11]{JK}) we conclude that there exists an $(L',r_0)$-NTA-domain $\Omega'$ such that 
\begin{equation}
	\label{capdefinition} \Omega \cap B_{2} \subset \Omega' \subset \Omega \cap B_{2 L^2}\,, 
\end{equation}
where the constants $L',r_0$ depend only on $L \geq 2$ and the dimension $N$. As in \cite{JK} we call such a domain a \emph{cap}. Let us also define the joint boundary of the domain and the cap, $\Gamma :=\partial \Omega' \cap\partial \Omega$ and the retracted caps 
\begin{equation}
	\label{B plus} \Omega'_{s} := \Omega' \cap \{x: d(x,\Gamma) \geq s \}. 
\end{equation}
The reason for introducing the cap $\Omega'$ is the following useful property. 
\begin{lemma}
	\label{geometric lemma1} Let $x \in \Omega'_{s} \setminus \Omega'_{2s}$ for $s \leq \tilde s := r_0/(2L')$. Then there is a constant $C = C(L') > 1$ such that 
	\begin{equation}
		\nonumber d(x,\Omega'_{2s}) \leq C s\,. 
	\end{equation}
\end{lemma}
\begin{proof}
	Let $y \in \Gamma$ such that $d(x,y)=d(x,\Gamma)$, since $\Omega'$ is an NTA-domain there exists an interior corkscrew point $z = a_{2 L' s}(y) \in \Omega'$ such that 
	\begin{equation}
		\nonumber d(z,
		\partial \Omega') > 2s\,, 
	\end{equation}
	i.e. $z \in \Omega'_{2s}$. Furthermore by the triangle inequality 
	\begin{equation}
		\nonumber |z-x| \leq |z-y|+|y-x| \leq 2L's + 2s \equiv C(L')s, 
	\end{equation}
	which proves the lemma. 
\end{proof}

For the next lemma we denote 
\begin{equation*}
	M_s := \sup_{\Omega'_{s}} u, 
\end{equation*}
where the retracted cap $\Omega'_{s}$ is defined in \cref{B plus}. As in the previous section we assume that $u$ is a solution of \cref{rescaledPDE,rescaled-non-homogeneity}. Again we write $\phiRb(t) = \eta_R(t)t$ with $\eta_R(t)= R \eta(t) +1$.
\begin{lemma}
	[Existence of a critical value] \label{blow up A} Let $\Omega$ and $\Omega'$ be as above, and denote by $A = a_1(0)$ a corkscrew point for the origin $0 \in 
	\partial \Omega$. Assume that $u\in C(\Omega \cap B_{2 L^3})$ is a non-negative solution of \cref{rescaledPDE}, \cref{rescaled-non-homogeneity} and denote $M = \sup_{\Omega'} u$. For every $\delta>0$ and $\alpha \in (0,1)$ there are $S \in (0,\tilde s]$ ($\tilde s$ is from \Cref{geometric lemma1}) and a constant $C$ depending on $\delta$ and $\alpha$ but not on $u$ and $R$ such that either 
	\begin{equation*}
		\int_{u(A)}^{M} \frac{dt}{\phiRb(t)} \leq C 
	\end{equation*}
	or 
	\begin{equation}
		\label{alt2} \int_{u(A)}^{M_S} \frac{dt}{\phiRb(t)} \leq C \qquad \text{and} \qquad S^{\alpha} \, \eta_R (M_S) \leq \delta\,. 
	\end{equation}
	If \cref{osgood2} holds then we always have \cref{alt2}. 
\end{lemma}

Before the proof we would like to point out that the properties in \cref{alt2} for $M_S$ are exactly what we want for the critical value. The second inequality in \cref{alt2} says that the non-homogeneous part in the Harnack inequality in \Cref{corHarnack} is as small as we want. The first inequality says that we are still able to control the critical value $M_S$ in a precise way. 
\begin{proof}
	For the rest of this proof denote $\Omega_L = \Omega \cap B_{2 L^3}$. Let us fix $\delta>0$. To be able to use \Cref{geometric lemma1}, we define for every $s \in [0,\tilde s]$ (where $\tilde s$ is from \Cref{geometric lemma1}) 
	\begin{equation*}
		M_s := \sup_{\Omega'_{s}} u\,. 
	\end{equation*}
	
	Let us first assume that there exists $s \in (0,\tilde s]$ such that $s^{\alpha}\, \eta_R\left( M_s\right) \leq \delta$ and prove that this implies the second statement of the lemma. We define $S \in (0,\tilde s]$ as 
	\begin{equation*}
		S:= \sup \Big\{ s \in [0,\tilde s] \, : \, s^{\alpha}\, \eta_R\left( M_s\right) \leq \delta \Big\}. 
	\end{equation*}
	We need to show that there exists a constant such that 
	\begin{equation}
		\label{bound for S} \int_{u(A)}^{M_S} \frac{dt}{\phiRb(t)} \leq C. 
	\end{equation}
	To this aim let $K \in \en$ be such that $2^{-K-1}\tilde s \leq S \leq 2^{-K}\tilde s$. For every $k \leq K$ we define $s_k = 2^{-k}\tilde s$ and $s_{K}:= S$. Moreover we denote 
	\begin{equation*}
		M_k := \sup_{\Omega'_{s_{k}}} u. 
	\end{equation*}
	We claim that for every $k \leq K-1$ the following holds 
	\begin{equation}
		\label{for k less than K} \int_{M_{k}}^{M_{k+1}} \frac{dt}{\phiRb(t)} \leq C 2^{-k \alpha} 
	\end{equation}
	for a constant $C > 1$ depending on the NTA constant $L$. Note that $M_S = M_{K}$.
	
	Let us fix $k \leq K-1$. Let $x_{k+1}$ be a point in $\overline{\Omega'}_{s_{k+1}}$ such that $u(x_{k+1})=M_{k+1}$. By \Cref{geometric lemma1} we deduce that there exists a point $\tilde x \in \Omega'_{s_{k}}$ such that 
	\begin{equation}
		\label{closest point NTA} |x_{k+1} - \tilde x| \leq C s_{k+1} 
	\end{equation}
	for a constant $C$ depending on $L$. Moreover $d(x_{k+1},\Gamma) \geq s_{k+1}$. It can be seen from \Cref{def:NTA} (with respect to $\Omega$) that we may construct a sequence of equi-sized balls $B^j$, $j=1,\ldots,n$ with radii $\rho \approx s_k$ such that $2 B^j \subset \Omega_L$, pairwise intersecting and 
	\begin{equation}
		\nonumber x_{k+1} \in B^1, \quad \tilde x \in B^n\,, 
	\end{equation}
	where $n$ depends only on the NTA constants of $\Omega$. Let us now use \Cref{corHarnack} in each ball, and denoting 
	\begin{equation}
		\nonumber \bar m_j = \inf_{B^j} u, \quad \bar M_j = \sup_{B^j} u, \quad j=1,\ldots, n\,, 
	\end{equation}
	we get 
	\begin{equation}
		\nonumber \sum_{j=1}^n \int_{\bar m_j}^{\bar M_j} \frac{dt}{\rho^\alpha \phiRb(t) +t} \leq n\, C\,. 
	\end{equation}
	Since the balls are pairwise intersecting we get 
	\begin{equation}
		\nonumber \int_{\min_j \bar m_j}^{\max_j \bar M_j} \frac{dt}{\rho^\alpha \phiRb(t)+t} \leq n\, C\,. 
	\end{equation}
	Now note two things. First, $M_k < \max_j \bar M_j$ since $x_{k+1} \in B^1$ and second, $M_{k-1} > \min_j \bar m_j$ since $\tilde x \in B^n \cap \Omega'_{s_k}$. Thus 
	\begin{equation}
		\nonumber \int_{M_{k}}^{M_{k+1}} \frac{dt}{\rho^\alpha \phiRb(t)+t} \leq \int_{\min_j \bar m_j}^{\max_j \bar M_j} \frac{dt}{\rho^\alpha \phiRb(t)+t} \leq n\, C\,. 
	\end{equation}
	Consequently for a constant $C_0 > 1$ depending only on $\alpha$ and on the NTA constants of $\Omega$ we have 
	\begin{equation}
		\label{from weak harnack} \int_{M_{k}}^{M_{k+1}} \frac{dt}{ (s_{k}^\alpha\eta_R(t) +1)t } \leq C_0. 
	\end{equation}
	
	Let us next show that 
	\begin{equation}
		\label{joku juttu} s_{k}^\alpha\, \eta_R (t) \geq \delta \qquad \text{for all }\, M_{k} < t < M_{k+1}. 
	\end{equation}
	Indeed, by the definition of $S$ we know that the following holds 
	\begin{equation*}
		s^\alpha\, \eta_R ( M_s ) \geq \delta \qquad \text{for all }\, s_{k+1} < s < s_{k}. 
	\end{equation*}
	Let us fix $t \in (M_k,M_{k+1})$. By continuity $t = M_s$ for some $s \in (s_{k+1},s_k)$. Therefore
	\[ s_k^{\alpha} \eta_R(t) \geq s^\alpha \eta_R(M_s) \geq \delta \]
	which proves \cref{joku juttu}.
	
	Finally we have by \cref{from weak harnack} and \cref{joku juttu} that 
	\begin{equation*}
		\left(\frac{\delta}{2 } \right) \int_{M_{k}}^{M_{k+1}} \frac{dt}{s_{k}^\alpha \eta_R(t) t} \leq \int_{M_{k}}^{M_{k+1}} \frac{dt}{ (s_{k}^\alpha \eta_R(t) + 1)t} \leq C_0. 
	\end{equation*}
	This proves \cref{for k less than K} since $\phiRb(t)= \eta_R(t)t$ and $s_k \leq 2^{-k}$.
	
	Summing \cref{for k less than K} over $k = 0, \dots, K-1$ we conclude that there is a constant $C$ such that 
	\begin{equation*}
		\int_{M_0}^{M_S} \frac{dt}{\phiRb(t)} = \sum_{k=1}^K \int_{M_{k-1}}^{M_{k}}\frac{dt}{\phiRb(t)} \leq C \sum_{k=1}^K 2^{-\alpha k} \leq C. 
	\end{equation*}
	Recall that by definition $M_0 =\sup_{\Omega'_{\tilde s}} u$ where $\Omega'_{\tilde s} := \Omega' \cap \{x: d(x,\Gamma) \geq \tilde s \}$. Therefore it follows from the fact that $\Omega$ is an NTA-domain together with repeated use of the interior Harnack (\Cref{corHarnack} with equi-sized balls) as before (staying inside $\Omega_L$) that 
	\begin{equation}
		\label{the first step} \int_{u(A)}^{M_0} \frac{dt}{\phiRb(t)} \leq C 
	\end{equation}
	and \cref{bound for S} follows.
	
	We need to treat the case when $s^{\alpha}\, \eta_R\left( M_s\right) > \delta$ for all $s \in (0,1]$. We show that this implies the first claim of the lemma. We define $M_k$ as before but now $K = \infty$. In this case we argue exactly as above and observe that \cref{for k less than K} holds for every $k \in \en$. Since $M = \sup_{\Omega'} u = \lim_{k \to\infty }M_k$ we obtain 
	\begin{equation*}
		\int_{M_{0}}^{M} \frac{dt}{\phiRb(t)} = \sum_{k=0}^\infty \int_{M_{k}}^{M_{k+1}} \frac{dt}{\phiRb(t)} \leq C \sum_{k=0}^\infty 2^{-\alpha k} \leq C. 
	\end{equation*}
	Hence we have 
	\begin{equation}
		\label{second alternative} \int_{u(A)}^{M} \frac{dt}{\phiRb(t)} \leq C 
	\end{equation}
	by \cref{the first step}. 
	
	Finally we note that in the above case the assumption $s^{\alpha}\, \eta_R\left( M_s\right) > \delta$ for all $s \in (0,1]$ necessarily implies $M = \lim_{k \to\infty }M_k = \infty$. Therefore if $\phi$ satisfies \cref{osgood2} then \cref{second alternative} provides a contradiction and we are never in the case that $s^{\alpha}\, \eta_R\left( M_s\right) > \delta$ for all $s \in (0,1]$. 
\end{proof}

Next we show that if $\delta$ in \cref{alt2} is chosen small enough the solution $u$ will blow-up as the solution of the homogeneous equation. The reason for this is that we may choose $\delta$ so small that the non-homogeneous term in the Harnack inequality in \Cref{corHarnack} stays small all the way up to the boundary. This follows from the assumption that $\phi$ is of the form $\phi(t)= \eta(t)t$, where $\eta$ is a slowly increasing function We continue to use the notation $M_s := \sup_{\Omega'_{s}} u$, where the retracted cap $\Omega'_{s}$ is defined in \cref{B plus}, $\phiRb(t) = R\phi(t) +t$ and use the notation $\phiRb(t) = \eta_R (t)t$, i.e., $\eta_R(t)= R \eta(t)+1$. 
\begin{lemma}
	[Blow up rate after the critical value] \label{blow up b} Let $\Omega$, $\Omega'$, $u, \alpha$ and $S$ be as in \Cref{blow up A}. There exists a $\delta>0$ such that if for some $S \in (0,\tilde s]$ the following holds 
	\begin{equation*}
		S^{\alpha} \, \eta_R (M_S) \leq \delta, 
	\end{equation*}
	then for every $s \leq S $ 
	\begin{equation}
		\label{blow up b:a} M_s \leq \frac{C}{s^\gamma} M_S \qquad \text{and} \qquad s^\alpha \, \eta_R(M_s) \leq C 
	\end{equation}
	holds for some $\gamma>1$, where $\delta, \gamma$ and $C$ depends on $\alpha$ but not on $u$ and $R$. 
\end{lemma}
\begin{proof}
	First, let us choose $\delta_1>0$ such that 
	\begin{equation}
		\label{choice of delta_1} \Lambda_0 \eta (e^{2 C_0}) < \frac{1}{\delta_1} . 
	\end{equation}
	Here $\Lambda_0$ is the constant from the assumption \cref{P3} and $C_0$ is from \cref{from weak harnack}. Second, we choose $\delta_2>0$ such that 
	\begin{equation}
		\label{choice of delta_2} \Lambda_0 \, \sup_{k \in \en} 2^{-\alpha k} \, \eta (e^{2C_0k}) < \frac{1}{\delta_2}. 
	\end{equation}
	This is possible since $\eta$ is slowly increasing which implies $\eta (t) \leq C_\ep t^{\ep}$ for $t \geq 1$ for all $\ep>0$. We choose $\delta>0$ as 
	\begin{equation*}
		\delta := \delta_1 \delta_2 
	\end{equation*}
	and assume that $S^{\alpha} \, \eta_R (M_S) \leq \delta $. 
	
	Denote 
	\begin{equation*}
		s_k := 2^{-k} S \qquad \text{ and} \qquad M_k := \sup_{\Omega'_{s_k}} u. 
	\end{equation*}
	First we prove that for every $k \in \en$ the following holds: 
	\begin{equation}
		\label{bound for maximum} s_k^{\alpha} \eta_R (M_k) \leq \delta_1 \qquad \text{implies} \qquad s_{k}^{\alpha} \eta_R(t) < 1 \,\, \text{for all } \, t \in [M_k, M_{k+1}]. 
	\end{equation}
	We argue by contradiction and assume that the implication \cref{bound for maximum} is not true. Let $T \in (M_k, M_{k+1}]$ be the first number for which 
	\begin{equation}
		\label{conradiction 1} s_{k}^{\alpha} \eta_R (T) = 1. 
	\end{equation}
	Since $\eta \geq 1$ is non-increasing on $(0,1)$ then necessarily $T\geq 1$. Moreover, since we assume $s_k^{\alpha} \eta_R (M_k) \leq \delta_1$ then we have $s_{k}^{\alpha} \eta_R (t) \leq 1$ for all $M_k < t <T$. As in the proof of \Cref{blow up A} we choose $x_{k+1} \in \overline{\Omega'}_{s_{k+1}}$ such that $M_{k+1} = u(x_{k+1})$ and let $\tilde x \in \Omega'_{s_{k}}$ be such that \cref{closest point NTA} holds (see \Cref{geometric lemma1}). Then we can proceed as in \cref{from weak harnack} in \Cref{blow up A} to conclude that 
	\begin{equation*}
		C_0 \geq \int_{M_{k}}^{M_{k+1}} \frac{dt}{(s_{k}^\alpha \eta_R(t) +1)\, t} \geq \int_{M_k}^{T} \frac{dt}{2t}\,. 
	\end{equation*}
	This implies $T \leq e^{2 C_0}M_k$. Since $T\geq 1$ and since $\eta_R$ is non-decreasing on $[1,\infty)$, we have by the assumptions \cref{P3'} on $\eta_R$ that 
	\begin{equation*}
		s_{k}^{\alpha} \eta_R(T) \leq s_k^{\alpha} \eta_R(e^{2 C_0}M_k) \leq \Lambda_0 \, \eta(e^{2 C_0}) \, s_k^{\alpha} \eta_R(M_k) <1 
	\end{equation*}
	by \cref{choice of delta_1}. This contradicts \cref{conradiction 1} and therefore \cref{bound for maximum} holds. 
	
	Recall that by our notation $M_S = M_0$ and $S = s_0$. We prove the first estimate in \cref{blow up b:a} by induction and claim that for every $k \in \en$ it holds that 
	\begin{equation}
		\label{induction 1} M_k \leq e^{2 C_0k} M_0. 
	\end{equation}
	Clearly \cref{induction 1} holds for $k=0$. Let us make the induction assumption that 
	\begin{equation}
		\label{indass1} \text{\cref{induction 1} holds true for $k > 0$ .} 
	\end{equation}
	First, by the assumptions on $u$ we have 
	\begin{equation*}
		s_0^{\alpha} \, \eta_R(M_0) \leq \delta. 
	\end{equation*}
	Let us show that we have 
	\begin{equation}
		\label{induction 2} s_{k}^{\alpha}\eta_R(t) <1, \qquad \text{for all } \, t \in [M_k, M_{k+1}]. 
	\end{equation}
	If $M_k < 1$ then since $\eta_R$ is non-increasing on $(0,1)$ we have 
	\begin{equation*}
		s_{k}^{\alpha}\eta_R(M_k) \leq s_{0}^{\alpha}\eta_R(M_{0}) \leq \delta < \delta_1 
	\end{equation*}
	and \cref{induction 2} follows from \cref{bound for maximum}. If $M_k \geq 1$ then by the induction assumption and by the assumptions \cref{P3'} on $\eta_R$ we have 
	\begin{equation*}
		\begin{split}
			s_k^\alpha \eta_R(M_k) &\leq s_k^\alpha\eta_R(e^{2 C_0k} M_0 ) \\
			&\leq \Lambda_0 2^{-\alpha k} \eta(e^{2 C_0k} ) \, s_0^\alpha \eta_R(M_0) \\
			&\leq \Lambda_0 2^{-\alpha k} \eta(e^{2 C_0k} ) \, \delta \leq \delta_1, 
		\end{split}
	\end{equation*}
	where the last inequality follows from \cref{choice of delta_2} and from the choice of $\delta$. Hence \cref{induction 2} follows from \cref{bound for maximum}.
	
	We need to show 
	\begin{equation}
		\label{induction 3} M_{k+1} \leq e^{2 C_0(k+1)} M_0. 
	\end{equation}
	Again arguing by iterating \Cref{corHarnack} as in \cref{from weak harnack} we may conclude that 
	\begin{equation*}
		\int_{M_k}^{M_{k+1}} \frac{dt}{(s_{k}^{\alpha} \eta_R(t) + 1)\, t} \leq C_0. 
	\end{equation*}
	Therefore by \cref{induction 2} we have 
	\begin{equation*}
		\int_{M_k}^{M_{k+1}} \frac{dt}{2 t} \leq C_0. 
	\end{equation*}
	We integrate the above inequality and use the induction assumption \cref{indass1} to deduce 
	\begin{equation*}
		M_{k+1} \leq e^{2 C_0}M_k \leq e^{2C_0(k+1)}M_0 
	\end{equation*}
	which proves \cref{induction 3}. Thus we have showed that \cref{indass1} implies \cref{induction 2} and \cref{induction 3} for $k+1$ and thus \cref{induction 1} holds for all $k \geq 0$, which implies the first estimate in \cref{blow up b:a}. The second estimate in \cref{blow up b:a} follows from \cref{induction 2}. 
\end{proof}

Using the $\delta \in (0,1)$ given by \Cref{blow up b} in \Cref{blow up A} we get the following result. 
\begin{theorem}
	\label{blow up 2} Let $\Omega$ and $\Omega'$ be as in the beginning of the section, and denote by $A = a_1(0)$ the corkscrew point for the origin $0 \in 
	\partial \Omega$. Assume that $u\in C(\Omega \cap B_{2 L^3})$ is a non-negative solution of \cref{rescaledPDE}, \cref{rescaled-non-homogeneity} and denote $M = \sup_{\Omega'} u$. Let $\alpha \in (0,1)$, then there is a constant $C_2(\alpha) > 1$ such that either 
	\begin{equation}
		\label{S0} \int_{u(A)}^{M} \frac{dt}{\phiRb(t)} \leq C_2, 
	\end{equation}
	or there is an $S \in (0,\tilde s]$ such that 
	\begin{equation}
		\label{S1} \int_{u(A)}^{M_S} \frac{dt}{\phiRb(t)} \leq C_2 \,, 
	\end{equation}
	\begin{equation}
		\label{S2} M_s \leq \frac{C_2}{s^\gamma} M_S \qquad \text{for every }\, s \in (0,S) \,, 
	\end{equation}
	and 
	\begin{equation}
		\label{S3} s^{\alpha}\, \eta_R(M_s) \leq C_2 \qquad \text{for every }\, s \in (0,S)\,. 
	\end{equation}
	However, if $\phi$ satisfies \cref{osgood2}, then \cref{S1}, \cref{S2}, and \cref{S3} always hold. 
\end{theorem}

\section{Proof of Theorem \ref{mainthm}} \label{secproof}

This section is devoted to the proof of the main theorem. 

\subsubsection*{Reduction} As discussed in \Cref{ssec reduction} we may assume that $\Omega$ is a Lipschitz domain with constant $0 < l < 1$ small enough so that $\Omega$ is Reifenberg flat with $\delta \leq \frac{1}{100}$. Assume that $0 \in\partial \Omega$. Furthermore, again alluding to \Cref{ssec reduction} we will assume that $u \in C(B_{16} \cap \overline{\Omega})$ is a non-negative solution of \cref{rescaledPDE,rescaled-non-homogeneity}. Due to the above assumption that $\Omega$ is a Lipschitz domain with constant $l$ and the assumption on scale (\Cref{ssec reduction}), we conclude that $\Omega$ is a $(2,16)$-NTA-domain.

\subsubsection*{Setup} Let us assume that we are in the case in \Cref{blow up 2} that \cref{S1}, \cref{S2} and \cref{S3} hold. Indeed, if we have \cref{S0} then the claim is trivially true.

Denote $M := \sup_{B_1 \cap \Omega} u$. Let $C_2$ be the constant from \Cref{blow up 2} as well as the values $S, M_S$, and the cap $\Omega'$. Note that $B_2 \cap \Omega \subset \Omega'$ by \cref{capdefinition}. We wish to prove that there is $\hat C>1$ such that $M \leq \hat C\, M_S$. This will prove the claim since $\phiRb(t) \geq t$ and therefore 
\begin{equation*}
	\int_{u(A)}^{M} \frac{dt}{\phiRb(t)} \leq \int_{u(A)}^{M_S} \frac{dt}{\phiRb(t)} + \int_{M_S}^{\hat C M_S} \frac{dt}{t} \leq C_2 + \log \hat C. 
\end{equation*}

\subsubsection*{Contradiction argument} Assume that there exists a point $P_1 \in B_1 \cap \Omega$ such that 
\begin{equation}
	\label{counterassumption} u(P_1) > \hat C\, M_S. 
\end{equation}
We will in the following use the short notation $d(x) = d(x,\partial \Omega) \leq d(x,\Gamma)$ where $\Gamma=\partial \Omega \cap\partial \Omega'$. By the definition of $M_S = \sup_{\Omega'_{S}} u$ we have $d(P_1) < S$. Therefore \Cref{blow up 2} yields 
\begin{equation}
	\label{P1 blow up} u(P_1) \leq \frac{C_2}{d(P_1)^\gamma} M_S \qquad \text{and} \qquad d(P_1)^{\alpha}\, \eta_R\left(u(P_1)\right) \leq C_2, 
\end{equation}
where $\alpha$ is from \Cref{bdry holder cont}. To show the second statement above, note that by continuity there is an $s_1$ such that $M_{s_1} = u(P_1)$ and $d(P_1) \leq s_1$. Thus for this particular $s_1$ we get from \Cref{blow up 2} that 
\begin{equation}
	\nonumber s_1^{\alpha}\, \eta_R\left(M_{s_1}\right) \leq C_2 
\end{equation}
which gives the statement. From \cref{counterassumption} we conclude 
\begin{equation}
	\label{mainthm step1} d(P_1) \leq \left( \frac{C_2}{\hat C} \right)^{\frac{1}{\gamma}}=: d_1 \qquad \text{and} \qquad d(P_1)^{\alpha}\, \eta_R\left(u(P_1)\right) \leq C_2. 
\end{equation}

Let $k > C_1$ be a number such that 
\begin{equation}
	\label{first holder small} C_1 \, k^{-\alpha} < 2^{-\gamma-1}, 
\end{equation}
where $C_1$ and $\alpha$ are from \Cref{bdry holder cont} and $\gamma>1$ is from \Cref{blow up 2}. Moreover, by choosing $\hat C$ in \cref{counterassumption} large enough we may assume that $d_1$ is so small that 
\begin{equation}
	\label{second holder small} C_1\, C_2\, \Lambda_0 \, 2^{-\alpha}\,\phi(2^\gamma) d_1^{\alpha}< 2^{-\gamma-1} 
\end{equation}
and 
\begin{equation}
	\label{kd1 small} k d_1 < \frac{1}{2}, 
\end{equation}
where $C_2$ is from \Cref{blow up 2} and $\Lambda_0$ is from the assumption \cref{P3}. Let $\hat P_1$ be a point on $\partial \Omega$ such that $|P_1- \hat P_1 | = d(P_1)$. Let us show that there is a point $P_2 \in B_{kd_1}(\hat P_1)\cap \Omega$ such that 
\begin{equation}
	\label{to show main} u(P_2) \geq 2^{\gamma} u(P_1), \qquad d(P_2) \leq \frac{d_1}{2}, \qquad \text{and} \qquad d(P_2)^{\alpha}\, \eta_R\left(u(P_2)\right) \leq C_2. 
\end{equation}

First we use \Cref{bdry holder cont} in $B_{kd_1}(\hat P_1) \cap \Omega$ with $r = kd_1$ and $\rho = d(P_1) \leq d_1$, and conclude that there is a point $P_2 \in B_{kd_1}(\hat P_1) \cap \Omega$ such that $u(P_2)\geq u(P_1)$ and by \cref{kd1 small} we have 
\begin{align}
	\label{from Holder cont} \sup_{B_\rho(\hat P_1) \cap \Omega} u &\leq C_1 u(P_2) k^{-\alpha} + C_1\phiRb( u(P_2)) \, (k d_1)^{2 \alpha} d(P_1)^{2\alpha} \\
	&\leq C_1 u(P_2) k^{-\alpha} + C_1\phiRb( u(P_2)) \, 2^{-\alpha} d(P_1)^{2\alpha}\,. \nonumber 
\end{align}

Let us show the first claim in \cref{to show main}, i.e., 
\begin{equation*}
	u(P_2) \geq 2^{\gamma} u(P_1). 
\end{equation*}
We argue by contradiction and assume that 
\begin{equation*}
	u(P_2) <2^{\gamma} u(P_1). 
\end{equation*}
Then we have by the assumption \cref{P3'} on $\eta$ that 
\begin{equation*}
	\begin{split}
		\phiRb(u(P_2)) \leq \phiRb(2^{\gamma} u(P_1)) &= 2^{\gamma} \eta_R(2^{\gamma} u(P_1)) \, u(P_1)\\
		&\leq \Lambda_0 (2^{\gamma}\eta(2^{\gamma})) \, \eta_R( u(P_1))u(P_1) \\
		&= \Lambda_0 \phi(2^{\gamma}) \, u(P_1) \, \eta_R( u(P_1)) . 
	\end{split}
\end{equation*}
Therefore by \cref{from Holder cont}, the above inequality, \cref{first holder small} and \cref{mainthm step1}, and finally by \cref{second holder small} we conclude 
\begin{equation*}
	\begin{split}
		u(P_1) \leq \sup_{ B_\rho(\hat P_1) \cap \Omega} u &\leq C_1 u(P_2) k^{-\alpha} + C_1 2^{-\alpha} \phiRb(u(P_2)))d(P_1)^{2\alpha} \\
		&\leq C_1 u(P_2) k^{-\alpha} + C_1 \Lambda_0 2^{-\alpha}\phi(2^{\gamma})d_1^\alpha\, u(P_1)\, \eta_R( u(P_1)) d(P_1)^\alpha\\
		&\leq 2^{-\gamma-1} u(P_2)+ C_1C_2 \Lambda_0 2^{-\alpha}\phi(2^\gamma) d_1^{\alpha} \, u(P_2) \\
		&\leq 2^{-\gamma} \,u(P_2). 
	\end{split}
\end{equation*}
This contradicts $u(P_2) <2^{\gamma} u(P_1)$ and thus the first claim in \cref{to show main} is proved. To continue we use the same argument as in \cref{P1 blow up} by applying \Cref{blow up 2} to get 
\begin{equation*}
	u(P_2) \leq \frac{C_2}{d(P_2)^\gamma}M_S, \qquad \text{and} \qquad d(P_2)^{\alpha}\, \eta_R\left(u(P_2)\right) \leq C_2, 
\end{equation*}
which proves the third claim. Since $u(P_2) \geq 2^{\gamma} u(P_1)$ we deduce 
\begin{equation*}
	\hat C M_S \leq u(P_1) \leq 2^{-\gamma} u(P_2) \leq 2^{-\gamma} \, \frac{C_2}{d(P_2)^\gamma}M_S. 
\end{equation*}
This implies 
\begin{equation*}
	d(P_2) \leq \frac{1}{2} \left( \frac{C_2}{\hat C} \right)^{\frac{1}{\gamma}} = \frac{d_1}{2}. 
\end{equation*}
Hence we have proved \cref{to show main}.

We may repeat the argument for \cref{to show main} we find a sequence of points $(P_i)$ such that $P_{i} \in B_{k d_{i-1}}(\hat P_{i-1}) \cap \Omega$, 
\begin{equation}
	\label{u blows up} u(P_{i}) \geq 2^{\gamma} u(P_{i-1}) \qquad \text{and} \qquad d(P_i) := \dist(P_i, 
	\partial \Omega) \leq 2^{-i+1}d_1 =: d_{i} 
\end{equation}
for every $i = 2,3,\dots$. By construction for every $l \geq 2$ we have 
\begin{equation*}
	|P_{l} - P_1| \leq \sum_{i = 1}^{l-1} |P_{i+1} - P_i| \leq \sum_{i = 1}^{l-1} k d_{i} = 2 k d_1 \sum_{i = 1}^{l-1} 2^{-i} \leq 2 k d_1 <1. 
\end{equation*}
Since $P_1 \in B_1 \cap \Omega$ we have $P_i \in B_2 \cap \Omega$ for every $i \in \en$. Moreover 
\begin{equation*}
	\lim_{i \to \infty} \dist(P_i, 
	\partial \Omega) = 0. 
\end{equation*}
By \cref{u blows up} we deduce that 
\begin{equation*}
	\lim_{i \to \infty} u(P_i) = \infty 
\end{equation*}
which contradicts the fact that $u$ vanishes continuously on $\partial \Omega$. \qed

\section{The Boundary Harnack Principle} \label{secbhi}

In this section we use the Carleson estimate to prove a boundary Harnack principle for two non-negative solutions which vanish on the boundary (\Cref{the boundary Harnack}). The proof is based on barrier function estimate and this requires the boundary of the domain to satisfy exterior and interior ball condition, i.e., the boundary has to be $C^{1,1}$-regular. 

\subsection{\em Proof of Theorem \ref{the boundary Harnack}}

Since $\Omega$ is a $C^{1,1}$-domain we may, by flattening the boundary, rescaling (see \Cref{ssec reduction}) and translating, assume that $0 \in\partial \Omega$ and $\Omega \cap B(0,16C) = \rn_+ \cap B(0,16C)$, where $C$ is the constant in \Cref{mainthm}, and $u,v$ are solutions of \cref{rescaledPDE,rescaled-non-homogeneity}.

It is enough to show that 
\begin{equation}
	\label{ratio bounded} \sup_{ t \in (0,1) } \frac{v(z+te_N)}{u(z+te_N)} \leq \frac{\mu_1}{\mu_0} 
\end{equation}
for every $z \in B(0,C) \cap \{ x_n = 1 \}$ for numbers $\mu_0$ and $\mu_1$ which satisfy the bound stated in the theorem. In fact, it is enough to show \cref{ratio bounded} for $z = e_N$.

Denote $x_0 = -e_N$, $x_1 = 2e_N$ and $M_v = \sup_{B_3^+ } v$, $m_u = \inf_{B(x_1,1)}u$ where $B_3^+ = B_3 \cap \rn_+$. In particular, $m_u \leq M_v$. First, by \Cref{corHarnack} we deduce 
\begin{equation}
	\label{estimate m_u} \int_{m_u}^{u(A)} \frac{ds}{\phiRb(s)} \leq C. 
\end{equation}
Second, by \Cref{mainthm} we have 
\begin{equation}
	\label{estimate M_v} \int_{v(A)}^{M_v} \frac{ds}{\phiRb(s)} \leq C. 
\end{equation}
We divide the proof in two cases. First we assume that 
\begin{equation}
	\label{assume m_u} \int_0^{m_u/3} \frac{ds}{\phiRb(s)} \geq 4\tilde{C} 
\end{equation}
and 
\begin{equation}
	\label{assume M_v} \int_{M_v}^{\infty} \frac{ds}{\phiRb(s)} \geq 2\tilde{C} 
\end{equation}
holds, where $\tilde{C}$ is a large constant which we choose later. 

We construct two $C^2$-regular barrier functions $w_1, w_2$ such that 
\begin{equation}
	\label{model1.1} \mathcal{P}_{\lambda, \Lambda}^-(D^2 w_2 ) \geq 2 \phiRb(|Dw_2|), \qquad \text{in } V:= B(x_0, 3)\setminus \bar{B}(x_0,1) 
\end{equation}
and 
\begin{equation}
	\label{model2.1} \mathcal{P}_{\lambda, \Lambda}^+(D^2 w_1) \leq - 2\phiRb(|Dw_1|), \qquad \text{in } U:= B(x_1, 2)\setminus \bar{B}(x_1, 1). 
\end{equation}
Moreover $w_1, w_2$ are such that their gradient do not vanish and they have boundary values $w_2 \geq 0$ on $\partial B(x_0,1)$ and $w_2 = M_v$ on $\partial B(x_0, 3)$, and $w_1 = 0$ on $\partial B(x_1, 2)$ and $w_1 = m_u$ on $\partial B(x_1, 1)$. Hence we have that $w_2 \geq v$ on $\partial V$ and $w_1 \leq u$ on $\partial U$. Since $v$ is a viscosity subsolution of \cref{model2} and since $|Dw_2|>0$ it follows from \cref{model1.1} and the definition of viscosity subsolution that $v-w_2$ does not attain local maximum in $V$. Therefore we deduce that $w_2 \geq v$ in $V$. Similarly we have $w_1 \leq u$ in $U$. Thus it is enough to bound the ratio 
\begin{equation*}
	\sup_{ t \in (0,1) } \frac{w_2(te_N)}{w_1(te_N)}. 
\end{equation*}

To construct $w_1$ we define $g:(0,1) \to \er$ such that 
\begin{equation}
	\label{construct g} t = \int_{\mu_0}^{g(t)} \frac{ds}{\tilde{C} \phiRb(s) } \quad \text{for }\, t \in (0,1) 
\end{equation}
where $\tilde{C}>1$, which is the constant in \cref{assume m_u} and \cref{assume M_v}, and $0 < \mu_0 \leq m_u$ are constants which we choose later. Note that $g$ is well defined by the implicit function theorem due to \cref{assume M_v} (recall that $m_u \leq M_v$). Then we have $g(0)= \mu_0$ and $g' =\tilde{C} \phiRb(g)$. We define the lower barrier $w_1: U \to \er$ by 
\begin{equation*}
	w_1(x):= \int_0^{2-|x-x_1|} g(t)\, dt. 
\end{equation*}
Then $w_1 = 0$ on $\partial B(x_1, 2)$. If we choose $\mu_0 = 0$ in \cref{construct g} we deduce from \cref{assume m_u} that $g(t) \leq m_u/3$ for all $t \in (0,1)$. This implies $w_1(x) \leq m_u/3$ for all $x \in\partial B(x_1, 1)$. On the other hand, by choosing $\mu_0 = m_u$ in \cref{construct g} yields $g(t) > m_u$ for all $t \in (0,1)$, which implies $w_1(x) > m_u$ for all $x \in\partial B(x_1, 1)$. Hence, by continuity we may choose $0 < \mu_0 < m_u$ such that $w_1 = m_u$ on $\partial B(x_1, 1)$. Finally it follows from the construction that 
\begin{equation}
	\nonumber \label{} \inf_{U} |Dw_1| \geq \inf_{t \in (0,1)} g(t) \geq \mu_0 >0\,. 
\end{equation}

After a straightforward calculation we see that since $g' =\tilde{C} \phiRb(g)$, we may choose the constant $\tilde C>2$ in \cref{construct g} large enough such that $w_1$ satisfies the following inequality in $U$ 
\begin{equation*}
	\begin{split}
		\mathcal{P}_{\lambda, \Lambda}^+(D^2 w_1(x)) &= -\lambda \tilde{C} \phiRb(g(2-|x-x_1|)) +\frac{n-1}{|x-x_1|}\Lambda\, g(2-|x-x_1|) \\
		&\leq -2 \phiRb(g(2-|x-x_1|)) \\
		&=- 2 \phiRb(|Dw_1(x)|). 
	\end{split}
\end{equation*}
The inequality above follows from $\phiRb(t)\geq t$.

The upper barrier function $w_2$ is constructed similarly by defining first for every $\mu_1 \geq M_v/3$ a function $f :(0,3) \to \er$ as 
\begin{equation*}
	t = \int_{f(t)}^{\mu_1} \frac{ds}{\tilde{C}\phiRb(s)} \quad \text{for }\, t \in (0,2). 
\end{equation*}
This is well defined due to \cref{assume m_u}. For $x \in V$ we define $w_2(x)$ by 
\begin{equation*}
	w_2(x) := \int_0^{|x-x_0|-1} f(t)\, dt. 
\end{equation*}
Then $w_2 = 0$ on $\partial B(x_0,1)$. By choosing $\mu_1 = M_v/3$ gives $w_2(x) <M_v$ for all $x \in\partial B(x_0,3)$. Therefore, by continuity we may choose $\mu_1 \geq M_v/3$ such that $w_2 = M_v$ on $\partial B(x_0,3)$. Finally we choose $\tilde{C}$ so large that $w_2$ satisfies \cref{model2.1} in $V$. Note that it follows from \cref{assume m_u} that, $ \inf_{t \in (0,2)} f >0$. Hence we have 
\begin{equation}
	\nonumber \label{} \inf_{V} |Dw_2| \geq \inf_{t \in (0,2)} f >0\,. 
\end{equation}

To prove the claim we will show that 
\begin{equation}
	\label{bdry hr 1} \sup_{ t \in (0,1) } \frac{w_2(te_N)}{w_1(te_N)} \leq \frac{\mu_1}{\mu_0} 
\end{equation}
and that 
\begin{equation}
	\label{bdry hr 2} \int_{\mu_0}^{\mu_1} \frac{dt}{\phiRb(t)} \leq C. 
\end{equation}
Let us study the functions $\tilde{w}_1(t) =w_1(te_N)$ and $\tilde{w}_2(t) =w_2(te_N)$ for $t \in [0,1]$. By construction we have that $\tilde{w}_1'(t) =g(t)$ and $\tilde{w}_2'(t) =f(t)$. Since $g' \geq 2\phiRb(g)$ and $f' \leq - 2\phiRb(f)$ we conclude that $\tilde{w}_1$ is convex and $\tilde{w}_2$ is concave. In particular, for every $t \in (0,1)$ we have 
\begin{equation*}
	\tilde{w}_1(t) \geq \tilde{w}_1'(0)\, t = \mu_0 \, t\qquad \text{and} \qquad \tilde{w}_2(t) \leq \tilde{w}_2'(0)\, t = \mu_1 \,t. 
\end{equation*}
In particular, we have 
\begin{equation*}
	\sup_{ t \in (0,1) } \frac{\tilde{w}_2(t)}{\tilde{w}_1(t)} \leq \frac{\mu_1}{\mu_0}\,, 
\end{equation*}
which is \cref{bdry hr 1}. 

Recall that $\tilde{w}_1(0) = 0$ and $\tilde{w}_1(1) = m_u$. By the mean value theorem there exists $\xi \in (0,1)$ such that $m_u = \tilde{w}'_1(\xi) = g(\xi)$. By \cref{construct g} we have 
\begin{equation*}
	1 \geq \xi = \int_{\mu_0}^{g(\xi)} \frac{dt}{\tilde{C}\phiRb(t)} = \int_{\mu_0}^{m_u} \frac{dt}{\tilde{C}\phiRb(t)}. 
\end{equation*}
Similarly we deduce that 
\begin{equation}
	\nonumber \int_{M_v}^{\mu_1} \frac{dt}{\phiRb(t)} \leq 3\tilde{C}. 
\end{equation}
Since $u(A) = v(A)$ the estimate \cref{bdry hr 2} follows from the previous two inequalities, \cref{estimate m_u,estimate M_v}. 

We need to deal the case when either \cref{assume m_u} or \cref{assume M_v} does not hold. In this case the result is almost trivial since we do not claim that the ratio $v/u$ is bounded. Assume first that \cref{assume m_u} does not hold. Then we simply choose $\mu_0 = 0$ and $\mu_1 = u(A)$. The estimate \cref{bdry hr 2} follows from \cref{estimate m_u} as follows 
\begin{equation*}
	\int_{0}^{u(A)} \frac{dt}{\phiRb(t)} \leq \int_{0}^{m_u/3} \frac{dt}{\phiRb(t)} + \int_{m_u/3}^{m_u} \frac{dt}{t} + \int_{m_u}^{u(A)} \frac{dt}{\phiRb(t)} \leq 4\tilde{C}+ \log 3 + C. 
\end{equation*}
If \cref{assume M_v} does not hold, we choose $\mu_0 = v(A)$ and $\mu_1 = \infty$. Then by \cref{estimate M_v} we have 
\begin{equation*}
	\int_{v(A)}^\infty \frac{dt}{\phiRb(t)} \leq \int_{v(A)}^{M_v} \frac{dt}{\phiRb(t)} + \int_{M_v}^{\infty} \frac{dt}{t} \leq 2\tilde{C}+ C. 
\end{equation*}
\qed

\subsection{Example for the sharpness of the boundary Harnack principle}

Here we discuss the sharpness of \Cref{the boundary Harnack}. We will only consider the case of the $p(x)$-Laplace equation and show that \Cref{p(x) bHp} is sharp. To simplify the argument we construct the example for cubes in the plane. To this aim we construct two non-negative $p(x)$-harmonic functions in the cube $Q = (0,1)^2 \subset \er^2$ such that $v(x_c) \leq u(x_c)$ at the center point $x_c =(1/2,1/2)$ and $u,v = 0$ at the bottom of the cube $(0,1)\times \{0\}$. We will show that the ratio in a smaller cube $Q' = (1/8,7/8)\times (0,1/2)$
\[ \sup_{x \in Q'} \frac{v(x)}{u(x)} \]
is not uniformly bounded, but it depends on the value $u(x_c)$ as in the statement of \Cref{p(x) bHp}.

First let us choose $p(\cdot) \in C^\infty(Q)$ to be
\[ p(x) = 3 - x_1, \qquad \text{where $x = (x_1,x_2)$}\,. \]
Then the $p(x)$-Laplace equation \cref{px} in non-divergence form for smooth functions with non-vanishing gradient reads as 
\begin{equation}
	\label{px equation} -\Delta w - (1-x_1)\Delta_\infty w = -\log (|\nabla w|) \, w_{x_1}, 
\end{equation}
where $\Delta_\infty w = \big \langle D^2w \frac{Dw}{|Dw|}, \frac{Dw}{|Dw|} \big \rangle$ denotes the infinity Laplacian. The point is that the equation is homogeneous for functions of type $u(x) = f(x_2)$ and non-homogeneous for $u(x)= g(x_1)$.

Let $H_{min} > 10^4$ be a constant to be fixed, and consider $H \geq H_{min}$. We define $u \in C^2(Q)$ simply to be 
\begin{equation*}
	u(x) = 2 H x_2. 
\end{equation*}
Then $u$ is a solution of \cref{px equation} and satisfies $u(x_c) = H$ at the center point $x_c =(1/2,1/2)$ and $u(x) = 0$ when $x_2 = 0$. Let us construct a solution $v$ such that $v(x) = 0$ when $x_2 = 0$, 
\begin{equation}
	\label{condition 1 for v} v(x_c) \leq H 
\end{equation}
and at a point $\hat{x} = (7/8,1/2) \in \overline{Q'}$ the following holds 
\begin{equation}
	\label{condition 2 for v} v(\hat{x}) \geq H^{\gamma} 
\end{equation}
for some $\gamma>1$. This will prove that the ratio satisfies
\[ \sup_{x \in Q'} \frac{v(x)}{u(x)} \geq H^{\gamma-1} \]
since $u(\hat{x}) = H$, this implies that the power-like behavior observed in \Cref{p(x) bHp} is sharp.

To this aim we choose $K>1$ to be the number which satisfies 
\begin{equation}
	\label{choice fo K:M} H= \int_0^{1/2}e^{e^{K+s/2}}\,ds, \qquad \text{and define} \qquad M := \int_0^{1} e^{e^{K+s/2}}\,ds. 
\end{equation}
Note that since $e^{e^{K+s/2}}$ is increasing we have 
\begin{equation}
	\label{estimate for H} H \leq e^{e^{K+1/4}}\,. 
\end{equation}
We will need the following easy estimate.
\begin{lemma}
	\label{retarded estimate} Let $\eps \in (0,2^{-2})$ be fixed, then there exists a $\hat K(\eps)$ such that for all $K \geq \hat K$ the following holds 
	\begin{equation}
		\nonumber \int_0^{1} e^{e^{K+s/2}}\,ds \geq e^{e^{K+1/2-\eps}} 
	\end{equation}
\end{lemma}
\begin{proof}
	First let $K > 1$, and note that $f(s) = e^{e^{K+s/2}}$ is a strictly increasing function. Then from the mean value theorem we get 
	\begin{equation}
		\nonumber \int_{1-2\eps}^{1-\eps} e^{e^{K+s/2}} ds \geq \eps e^{e^{K+1/2-\eps}}\,. 
	\end{equation}
	It is now enough to show that 
	\begin{equation}
		\nonumber \frac{1}{\eps} \int_{1-2\eps}^{1-\eps} e^{e^{K+s/2}} ds \leq \int_{1-\eps}^{1} e^{e^{K+s/2}} ds. 
	\end{equation}
	A change of variables leads to 
	\begin{equation}
		\nonumber \frac{1}{\eps} \int_{1-2\eps}^{1-\eps} e^{e^{K+s/2}} ds \leq \int_{1-2\eps}^{1-\eps} e^{e^{K+s/2+\eps/2}} ds. 
	\end{equation}
	We now see that it is enough to prove the much stronger inequality 
	\begin{equation}
		\nonumber \frac{1}{\eps} \leq \left [e^{e^{K}} \right ]^{e^{\eps/2}-1} \,, 
	\end{equation}
	which is obviously true for a large enough $K(\eps)$ since $e^{\eps/2}-1 > 0$. 
\end{proof}
Let us denote $\Gamma = \{1\}\times (0,1) \subset\partial Q$. We choose $v$ to be the solution of the Dirichlet problem 
\begin{equation}
	\label{veq} 
	\begin{cases}
		&-\Delta v - (1-x_1)\Delta_\infty v = -\log (|\nabla v|) \, v_{x_1}, \\
		&v = M \,\, \text{on} \,\,\Gamma, \\
		&v = 0 \,\, \text{on} \,\,
		\partial Q \setminus \Gamma. 
	\end{cases}
\end{equation}
Let us show that $v$ satisfies \cref{condition 1 for v} and \cref{condition 2 for v}.

To show \cref{condition 1 for v} we construct a barrier function $\varphi$ such that $\varphi \geq v$ in $Q$ and $\varphi(x_c) = H$. We choose $\varphi(x) = F(x_1)$ where $F$ is an increasing and convex function which is a solution of
\[ F'' = \frac{1}{2}\log (F') \, F' \]
with $F(0)= 0$. We may solve the above equation explicitly by
\[ F(t) = \int_0^{t}e^{e^{K+s/2}}\,ds \]
for any $K \in \er$. When we choose $K$ as in \cref{choice fo K:M} we get $\varphi(x_c)= F(1/2)=H$. Moreover by \cref{choice fo K:M} it holds that $F(1)=M$ and therefore $\varphi \geq v$ on $\partial Q$. It is easy to see that $\varphi \in C^2(Q)$ satisfies
\[ -\Delta \varphi - (1-x_1)\Delta_\infty \varphi >-\log (|\nabla \varphi|) \, \varphi_{x_1} \qquad \text{in }\, Q. \]
Since $v$ is a solution of \cref{veq} $v-\varphi$ does not attain a local maximum. Thus from $\varphi \geq v$ on $\partial Q$ it follows that $\varphi \geq v$ in $Q$. Hence we have \cref{condition 1 for v}. 

To show \cref{condition 2 for v} we construct a barrier function $\psi$ in $D:= Q \cap B(x_b,1/2)$, where $x_b = (5/4,1/2)$, such that $\psi \leq v$ in $D$ and $\psi(\hat{x}) \geq H^\gamma$. To this aim we define $\psi(x) = G(|x-x_b|)$ where $G :[1/4,1/2] \to\er$ is decreasing, non-negative and convex function such that $G(1/2)=0$, $G(1/4)=M$ and 
\begin{equation}
	\label{equation for G} G'' \geq \log |G'| \, |G'| + 4|G'|. 
\end{equation}
Let us for a moment assume that such a function exists. The inequality \cref{equation for G} implies that $\psi \in C^2(D)$ satisfies
\[ 
\begin{split}
	-\Delta \psi(x) - (1-x_1)\Delta_\infty \psi(x) &= -(2-x_1)G''(|x- x_b|) -\frac{G'(|x- x_b|)}{|x- x_b|} \\
	&< -G''(|x- x_b|) +4|G'(|x- x_b|)| \\
	&\leq -\log |G'(|x- x_b|)| \, |G'(|x- x_b|)| \\
	&\leq -\log |D\psi(x)| \, \psi_{x_1}(x) 
\end{split}
\]
for every $x \in D$. Therefore $v-\psi$ does not attain local minimum in $D$. From the conditions $G(1/2)=0$ and $G(1/4)=M$ it follows that $\psi \leq v$ on $\partial D$. Therefore $\psi \leq v$ in $D$.

To find $G$, we denote $G' =g$ and define
\[ g(t):=-e^{e^{R-(1+\eps)t}}, \]
where the large parameter $R \in \er$ and the small parameter $\eps \in(0,1)$ are chosen later. When $R>1$ is large $G' =g$ satisfies \cref{equation for G}. To see this note that 
\begin{equation}
	\nonumber g' = (1+\eps)e^{R-(1+\eps)t}|g|. 
\end{equation}
Thus \cref{equation for G} becomes 
\begin{equation}
	\nonumber (1+\eps)e^{R-(1+\eps)t}|g| \geq \log(|g|)|g|+4|g|. 
\end{equation}
After rewriting, this becomes 
\begin{equation}
	\nonumber \eps e^{R-(1+\eps)t} \geq 4, 
\end{equation}
which is true if $R > \hat R(\eps)$. 

We define
\[ G(t) = M + \int_{1/4}^t g(s) \, ds\,. \]
Let us first choose $H_{min}(\eps)$, and thus $M$, large enough such that
\[ M > \int_{1/4}^{1/2} e^{e^{\hat R-(1+\eps)s}}\, ds > 0 \]
and choose $R\geq \hat R(\epsilon)$ such that
\[ G(1/2) = M- \int_{1/4}^{1/2} e^{e^{R-(1+\eps)s}}\, ds = 0. \]
Then we have $G(1/2)= 0$ and $G(1/4)=M$ as wanted. 

We need yet to show that $\psi(\hat{x}) \geq H^\gamma$. Recall that $\hat{x} = (7/8,1/2)$ and $x_b = (5/4,1/2)$. Hence $|\hat{x}-x_b|= 3/8$ and therefore $\psi(\hat{x}) = G(3/8)$. We use \Cref{retarded estimate}. \cref{choice fo K:M}, the definition of $G$ and the condition $G(1/2)=0$ to estimate 
\begin{equation}
	\label{estimate for R} e^{e^{K+1/2-\eps}} \leq \int_0^{1} e^{e^{K+s/2}}\,ds= M = \int_{1/4}^{1/2} e^{e^{R-(1+\eps)s}}\, ds \leq e^{e^{R-(1+\eps)/4}} 
\end{equation}
by possibly enlarging $H_{min}(\eps)$. Therefore we may estimate the value of $\psi$ at $\hat{x}$ by \cref{estimate for R} and \cref{estimate for H} and get
\[ 
\begin{split}
	\psi(\hat{x}) = G(3/8) -G(1/2) &= \int_{3/8}^{1/2} e^{e^{R-(1+\eps)s}}\, ds \geq \int_{3/8}^{7/16} e^{e^{R-(1+\eps)s}}\, ds \\
	&\geq c e^{e^{R-7/16(1+\eps)}} \geq c \left(e^{e^{R-(1+\eps)/4}} \right)^{e^{-3/16-\eps}}\\
	&\geq c \left(e^{e^{K+1/2-\eps}} \right)^{e^{-3/16-\eps}} \qquad (\text{by \cref{estimate for R}}) \\
	&= c \left(e^{e^{K+1/4 }} \right)^{e^{1/16-2\eps}}\\
	&\geq c H^{e^{1/16-2\eps}} \qquad (\text{by \cref{estimate for H}}). 
\end{split}
\]
We define $\gamma= e^{1/16-2 \eps}$ which is bigger than one by choosing $\eps>0$ small enough, hence also fixing $H_{min}$. This shows \cref{condition 2 for v}.


\begin{thebibliography}
	{99} \bibitem{AdLu} 
	\newblock \textsc{Adamowicz T., and Lundstr\"om N.L.P.}, 
	\newblock The boundary Harnack inequality for variable exponent $p$-Laplacian, Carleson estimates, barrier functions and $p(\cdot)$-harmonic measures. 
	\newblock \emph{Annali di Matematica Pura ed Applicata} (2015).
	
	\bibitem{Alk} 
	\newblock \textsc{Alkhutov Y.A.} 
	\newblock The Harnack inequality and the H\"older property of solutions of nonlinear elliptic equations with a nonstandard growth condition 
	\newblock \emph{Differential Equations,} \textbf{33} (1997), no. 12, 1651--1660.
	
	\bibitem{A1} 
	\newblock \textsc{Avelin B.} 
	\newblock On time dependent domains for the degenerate $p$-parabolic equation: Carleson estimate and H\"older continuity, 
	\newblock to appear in Mathematische Annalen. \href{http://dx.doi.org/10.1007/s00208-015-1226-8}{doi:10.1007/s00208-015-1226-8}
	
	\bibitem{AGS} 
	\newblock \textsc{Avelin B., Gianazza U. and Salsa S.} 
	\newblock Boundary Estimates for Certain Degenerate and Singular Parabolic Equations. 
	\newblock to appear in Journal of European Mathematical Society.
	
	\bibitem{ALuN} 
	\newblock \textsc{Avelin B., Lundstr{\"o}m N.L.P. and Nystr{\"o}m K.} 
	\newblock Boundary estimates for solutions to operators of p-Laplace type with lower order terms, 
	\newblock \emph{J. Differential Equations} \textbf{250} (2011), no. 1, 264--291.
	
	\bibitem{ALuN1} 
	\newblock \textsc{Avelin B., Lundstr{\"o}m N.L.P. and Nystr{\"o}m K.} 
	\newblock Optimal doubling, Reifenberg flatness and operators of p-Laplace type, 
	\newblock \emph{Nonlinear Anal.} \textbf{74} (2011), no. 17, 5943--5955.
	
	\bibitem{AN1} 
	\newblock \textsc{Avelin B. and Nystr{\"o}m K.} 
	\newblock Estimates for solutions to equations of $p$-Laplace type in Ahlfors regular NTA-domains. 
	\newblock \emph{J. Funct. Anal. } \textbf{266} (2014), no. 9, 5955--6005.
	
	\bibitem{BL} 
	\newblock \textsc{Bennewitz, B. and Lewis, J.L.} 
	\newblock On the dimension of p-harmonic measure. 
	\newblock \emph{Ann. Acad. Sci. Fenn. Math. } \textbf{30} (2005), no. 2, 459--505.
	
	\bibitem{BGT} 
	\newblock \textsc{Bingham N. H., Goldie C. M. and Teugels J. L.} 
	\newblock \emph{Regular variation.} 
	\newblock Encyclopedia of Mathematics and its Applications, 27. Cambridge University Press, Cambridge, (1989).
	
	\bibitem{CC} 
	\newblock \textsc{Caffarelli, L.A. and Cabre, X.} 
	\newblock \emph{Fully nonlinear Elliptic equations.} 
	\newblock American Mathematical Society Colloquium Publications, 43. American Mathematical Society, Providence, RI, (1995).
	
	\bibitem{CFMS} 
	\newblock \textsc{Caffarelli L., Fabes E., Mortola S. and Salsa S.} 
	\newblock Boundary behavior of non-negative solutions of elliptic operators in divergence form. 
	\newblock {\em Indiana Univ. Math. J.} \textbf{30} (1981), no. 4, 621--640.
	
	\bibitem{Carl} 
	\newblock \textsc{Carleson L.} 
	\newblock On the existence of boundary values for harmonic functions in several variables. 
	\newblock {\em Ark. Mat.} \textbf{4} (1962) 393--399.
	
	\bibitem{CFS} 
	\newblock \textsc{Cerutti M.C., Ferrari F. and Salsa, S.} 
	\newblock Two-phase problems for linear elliptic operators with variable coefficients: Lipschitz free boundaries are $C^{1,\gamma}$, 
	\newblock \emph{Arch. Ration. Mech. Anal.} \textbf{171} (2004), 329--448.
	
	\bibitem{CNP} 
	\newblock \textsc{Cinti C., Nystr\"om K. and Polidoro S.} 
	\newblock A Carleson-type estimate in Lipschitz type domains for non-negative solutions to Kolmogorov operators. 
	\newblock \emph{Ann. Sc. Norm. Super. Pisa Cl. Sci.,} (5) \textbf{12} (2013), no. 2, 439--465.
	
	\bibitem{CIL} 
	\newblock \textsc{Crandall M.G., Ishii H. and Lions P.-L.} 
	\newblock User's guide to viscosity solutions of second order partial differential equations. 
	\newblock \emph{Bull. Amer. Math. Soc.,} \textbf{27} (1992), 1--67.
	
	\bibitem{DeG} 
	\newblock \textsc{De Giorgi E.} 
	\newblock Sulla differenziabilità e l'analiticità delle estremali degli integrali multipli regolari. 
	\newblock \emph{Mem. Acc. Sc. Torino,} \textbf{3} (1957), 25--43.
	
	\bibitem{FaSa1} 
	\newblock \textsc{Fabes E.B. and Safonov M.}, 
	\newblock Behavior near the boundary of positive solutions of second order parabolic equations, 
	\newblock \emph{J. Fourier Anal. Appl.} \textbf{special issue} (1997), no. 3, 871--882.
	
	\bibitem{FaSa2} 
	\newblock \textsc{Fabes E.B., Safonov M. and Yuan Y.}, 
	\newblock Behavior near the boundary of positive solutions of second order parabolic equations. II, 
	\newblock \emph{Trans. AMS.} \textbf{351} (1999), no. 12, 4947--4961.
	
	\bibitem{Fe} 
	\newblock \textsc{Ferrari F.} 
	\newblock Two-phase problems for a class of fully non-linear elliptic operators, Lipschitz free boundaries are $C^{1,\gamma}$, 
	\newblock \emph{Amer. J. Math.} \textbf{128} (2006), 541--571.
	
	\bibitem{FS} 
	\newblock \textsc{Ferrari F. and Salsa S.} 
	\newblock Regularity of the free boundary in two-phase problems for linear elliptic operators, 
	\newblock \emph{Adv. Math.} \textbf{214} (2007), 288--322.
	
	\bibitem{FS1} 
	\newblock \textsc{Ferrari F. and Salsa S.}, 
	\newblock Subsolutions of elliptic operators in divergence form and applications to two-phase free boundary problems, 
	\newblock \emph{Bound. Value. Probl.} (2007), Art. ID 57049, 21pp.
	
	\bibitem{GT} 
	\newblock \textsc{Gilbarg D. and Trudinger N. S.} 
	\newblock Elliptic partial differential equations of second order, 
	\newblock \emph{second edition, Springer-Verlag,} (1983).
	
	\bibitem{HKL} 
	\newblock \textsc{Harjulehto P., Kinnunen J. and Lukkari T.} 
	\newblock Unbounded supersolutions of nonlinear equations with nonstandard growth 
	\newblock \emph{Bound. Value Probl.,} (2007), Art. ID 48348, 20 pp.
	
	\bibitem{JK} 
	\newblock \textsc{Jerison D. S. and Kenig C. E.} 
	\newblock Boundary behavior of harmonic functions in nontangentially accessible domains. 
	\newblock {\em Adv. Math.} \textbf{46} (1982), no. 1, 80--147.
	
	\bibitem{Jones} 
	\newblock \textsc{Jones P.}, 
	\newblock A geometric localization theorem, 
	\newblock \emph{Adv. in Math.} \textbf{46} (1982), no. 1, 71--79.
	
	\bibitem{Vesku} 
	\newblock \textsc{Julin V.}, 
	\newblock Generalized Harnack inequality for non-homogeneous elliptic equations, 
	\newblock \emph{Arch. Ration. Mech. Anal.} \textbf{216} (2015), no. 2, 673--702.
	
	\bibitem{JLP} 
	\newblock \textsc{Juutinen P. Lukkari T. and Parviainen M.}, 
	\newblock Equivalence of viscosity and weak solutions for the $p(x)$-Laplacian, 
	\newblock \emph{Ann. Inst. H. Poincaré Anal. Non Linéaire} \textbf{27} (2010), 1471--1487.
	
	\bibitem{K} 
	\newblock \textsc{Kenig C.} 
	\newblock Harmonic Analysis Techniques for Second Order Elliptic Boundary Value Problems, 
	\newblock \emph{CBMS Regional Conference Series in Math.}, vol. 83, AMS, Providence, RI, (1994).
	
	\bibitem{KT} 
	\newblock \textsc{Kenig C. and Toro T.} 
	\newblock Harmonic measure on locally flat domains 
	\newblock \emph{Duke Math. J.,}\textbf{87} (1997), 501--551.
	
	\bibitem{KS1} 
	\newblock \textsc{Krylov N.V. and Safonov M.V.} 
	\newblock An estimate of the probability that a diffusion process hits a set of positive measure. 
	\newblock \emph{Dokl. Akad. Nauk. SSSR} \textbf{245} (1979), 253--255. English translation in Soviet Math. Dokl. \textbf{20} (1979), 253--255.
	
	\bibitem{KS2} 
	\newblock \textsc{Krylov N.V. and Safonov M.V.} 
	\newblock Certain properties of solutions of parabolic equations with measurable coefficients. 
	\newblock \emph{Izv. Akad. Nauk SSSR} \textbf{40} (1980), 161--175. English translation in Math. SSSR Izv \textbf{161} (1980), 151--164.
	
	\bibitem{LLuN} 
	\newblock \textsc{Lewis J., Lundstr{\"o}m N.L.P. and Nystr{\"o}m K.} 
	\newblock Boundary Harnack Inequalities for Operators of $p$-Laplace type in Reifenberg Flat Domains, 
	\newblock \emph{in Perspectives in PDE, Harmonic Analysis, and Applications, Proceedings of Symposia in Pure Mathematics} \textbf{79} (2008), 229--266.
	
	\bibitem{LN1} 
	\newblock \textsc{Lewis J., and Nystr{\"o}m K.} 
	\newblock Boundary Behaviour and the Martin Boundary Problem for $ p $-Harmonic Functions in Lipschitz domains, 
	\newblock \emph{Ann. of Math.} \textbf{172} (2010), 1907--1948.
	
	\bibitem{LN5} 
	\newblock \textsc{Lewis J. and Nystr{\"o}m K.} 
	\newblock Regularity of Lipschitz Free Boundaries in Two-phase Problems for the $p$-Laplace Operator, 
	\newblock \emph{Adv. Math.} \textbf{225} (2010), 2565-2597.
	
	\bibitem{LN6} 
	\newblock \textsc{Lewis J. and Nystr{\"o}m K.} 
	\newblock Regularity of Flat Free Boundaries in Two-phase Problems for the $p$-Laplace Operator, 
	\newblock \emph{Ann. Inst. H. Poincar\'e Anal. Non Linéaire} \textbf{29} (2012), no. 1, 83--108. 
	
	\bibitem{Si} 
	\newblock \textsc{Sirakov S.} 
	\newblock Solvability of uniformly elliptic fully nonlinear PDE, 
	\newblock \emph{Arch. Ration. Mech. Anal.} \textbf{195} (2010), 579--607.
	
	\bibitem{Wo} 
	\newblock \textsc{Wolanski N.} 
	\newblock Local bounds, Harnack inequality and H\"older continuity for divergence type elliptic equations with non-standard growth, 
	\newblock Preprint (2013). 
\end{thebibliography}
\end{document}